\documentclass[a4paper,11pt,leqno]{article}

\usepackage[latin1]{inputenc}
\usepackage[T1]{fontenc} 
\usepackage[english]{babel}
\usepackage{verbatim}
\usepackage{calligra}
\usepackage{enumerate}
\usepackage{dsfont}
\usepackage{amsfonts}
\usepackage{amsmath}
\usepackage{amsthm}
\usepackage{amssymb}
\usepackage{mathtools}
\usepackage{mathrsfs}
\usepackage{color}
\usepackage{graphicx}
\usepackage{bm}
\usepackage{srcltx}

\usepackage{tikz}
\usepackage[retainorgcmds]{IEEEtrantools}

\usepackage{graphicx}		
\usepackage[hypcap=false]{caption}

\theoremstyle{definition}
\newtheorem{defi}{Definition}[section]

\newtheorem{rmk}[defi]{Remark}

\theoremstyle{plane}
\newtheorem{thm}[defi]{Theorem}

\newtheorem{lemma}[defi]{Lemma}

\newcommand{\tbf}{\textbf}

\newcommand{\tsl}{\textsl}

\newcommand{\mbb}{\mathbb}

\newcommand{\mc}{\mathcal}

\newcommand{\veps}{\varepsilon}

\newcommand{\what}{\widehat}
\newcommand{\wtilde}{\widetilde}
\newcommand{\vphi}{\varphi}
\newcommand{\oline}{\overline}

\newcommand{\ra}{\rightarrow}

\newcommand{\g}{\gamma}

\newcommand{\s}{\sigma}
\renewcommand{\t}{\tau}

\newcommand{\lam}{\lambda}
\newcommand{\de}{\delta}
\renewcommand{\o}{\omega}

\newcommand{\lan}{\langle}
\newcommand{\ran}{\rangle}

\newcommand{\R}{\mathbb{R}}

\newcommand{\N}{\mathbb{N}}
\newcommand{\Z}{\mathbb{Z}}
\newcommand{\T}{\mathbb{T}}

\newcommand{\D}{\mathbb{D}}

\renewcommand{\div}{{\rm div}\,}

\newcommand{\sign}{{\rm sign}}

\newcommand{\al}{\alpha}
\newcommand{\bt}{\beta}

\allowdisplaybreaks

\def\d{\partial}
\def\div{{\rm div}\,}
\newcommand{\dd}{\d^2_{x}}

\begin{document}

\newcommand{\fra}[1]{\textcolor{blue}{#1}}

\title{\textsc{\Large{\textbf{Well-posedness and singularity formation  
for the \\ Kolmogorov two-equation model of turbulence in 1-D}}}}

\author{\normalsize\textsl{Francesco Fanelli}$\,^{1a,b,c}\qquad$ and $\qquad$
\textsl{Rafael Granero-Belinch\'on}$\,^{2}$ \vspace{.5cm} \\
\footnotesize{$\,^{1a}\;$ \textsc{BCAM -- Basque Center for Applied Mathematics}}  \vspace{.1cm} \\
{\footnotesize Alameda de Mazarredo 14, E-48009 Bilbao, Basque Country, SPAIN} \vspace{.3cm} \\
\footnotesize{$\,^{1b}\;$ \textsc{Ikerbasque -- Basque Foundation for Science}}  \vspace{.1cm} \\
{\footnotesize Plaza Euskadi 5, E-48009 Bilbao, Basque Country, SPAIN} \vspace{.3cm} \\
\footnotesize{$\,^{1c}\;$ \textsc{Universit\'e de Lyon, Universit\'e Claude Bernard Lyon 1}}  \vspace{.1cm} \\
{\footnotesize \it Institut Camille Jordan -- UMR 5208}  \vspace{.1cm}\\
{\footnotesize 43 blvd. du 11 novembre 1918, F-69622 Villeurbanne cedex, FRANCE} \vspace{.3cm} \\
\footnotesize{$\,^{2}\;$ \textsc{Universidad de Cantabria}}  \vspace{.1cm} \\
{\footnotesize \it Departamento de Matem\'aticas, Estad\'istica y Computaci\'on}  \vspace{.1cm}\\
{\footnotesize Avda. Los Castros s/n, Santander, SPAIN} \vspace{.3cm} \\
\footnotesize{Email addresses: $\,^{1}\;$\ttfamily{ffanelli@bcamath.org}}, $\;$
\footnotesize{$\,^{2}\;$\ttfamily{rafael.granero@unican.es}}
\vspace{.2cm}
}

\date\today

\maketitle

\subsubsection*{Abstract}
{\footnotesize We study the Kolomogorov two-equation model of turbulence in one space dimension.
Two are the main results of the paper. First of all, we establish a local well-posedness theory in Sobolev spaces even in the case of vanishing mean turbulent kinetic
energy. Then, we show that there are smooth solutions which blow up in finite time.

To the best of our knowledge, these results are the first establishing the well-posedness of the system for vanishing initial data and the occurence of finite time singularities for the model under study.

}

\paragraph*{\small 2020 Mathematics Subject Classification:}{\footnotesize 35Q35 
(primary);
76F60, 
35B65, 
35B44 
(secondary).}

\paragraph*{\small Keywords: }{\footnotesize Kolmogorov two-equation model of turbulence; vanishing turbulent kinetic energy; local well-posedness; finite
time blow-up.}


\section{Introduction} \label{s:intro}

In 1941, in a series of papers Kolmogorov and Obukhov layed the foundations of the celebrated mathematical theory of turbulence
which now goes under the name of \emph{Komogorov's 1941 theory}.
We refer \tsl{e.g.} to Chapter 6 of \cite{Les} and to Chapters 4 and 14 of \cite{CR-B} for an account of that theory, as well as to references quoted therein.

As an outcome of that theory, in \cite{Kolm_42} (see the Appendix of \cite{Spald} for an English translation) Kolmogorov formulated
the following system of partial differential equations to describe a fully developed homogeneous isotropic turbulence:
\begin{equation} \label{eq:kolm_d}
\left\{\begin{array}{l}
       \d_tu\,+\,(u\cdot\nabla) u\,+\,\nabla\pi\,-\,\nu\,\div\left(\dfrac{k}{\o}\,\D u\right)\,=\,f \\[1ex]
       \d_t\o\,+\,u\cdot\nabla\o\,-\,\alpha_1\,\div\left(\dfrac{k}{\o}\,\nabla\o\right)\,=\,-\,\alpha_2\,\o^2 \\[1ex]
       \d_tk\,+\,u\cdot\nabla k\,-\,\alpha_3\,\div\left(\dfrac{k}{\o}\,\nabla k\right)\,=\,-\,k\,\o\,+\,\alpha_4\,\dfrac{k}{\o}\,\big|\D u\big|^2 \\[2ex]
       \div u\,=\,0\,.
       \end{array}
\right.
\end{equation}
In the previous system, the three unknowns $u$, $\o$ and $k$ are functions of time and space variables $(t,x)\in\R_+\times\Omega$, where $\Omega$ is a domain in
$\R^d$, with $d=2,3$. The vector field $u\in\R^d$ represents the average velocity of the fluid: it is the time-averaged velocity in the original paper by Kolmogorov
(see again \cite{Spald}), but it is interpreted as an ensemble average in the modern approach (see \cite{Mo-Pi}, \cite{Les}). The velocity
field $u$ is assumed to be incompressible, whence the last equation in \eqref{eq:kolm_d} and the presence of the lagrangian multiplier $\nabla\pi$ in
the first relation of the system, where $\pi$ is the mean pressure of the fluid.
On the other hand, $\o$ and $k$ are positive scalar functions, representing respectively the mean frequency of the turbulent fluctuations
and the average turbulent kinetic energy (called ``pulsation velocity'' in \cite{Spald}).
The vector field $f\in\R^d$ is a known (average) external force acting on the fluid, which we will set soon equal to $0$.
Finally, all the parameters $\nu,\al_1\ldots \al_4$ are physical adimensional parameters which are strictly positive numbers: we have
\[
\nu\,>0\,,\qquad \al_1\,,\;\al_2\,,\;\al_3\,,\;\al_4\;>\;0\,.
\]
In general, their values depend on the characteristics of the flow; however, for some of them Kolmogorov is able to specify the precise numerical value (we refer once more to \cite{Kolm_42} and \cite{Spald}).

From a mathematical point of view, system \eqref{eq:kolm_d} is a parabolic system, with viscosity (diffusion coefficient) proportional to $k/\o$. As a matter of fact, the symbol $\D$ appearing in the first and third equations stands for the symmetric part of the gradient of $u$:
\[
 \D u\,:=\,\frac{1}{2}\,\big(D u\,+\,\nabla u\big)\,,
\]
where we have denoted by $Du$ the Jacobian matrix of $u$ and by $\nabla u$ its transpose matrix. 
Notice that, despite the presence of the two damping terms $-\al_2\,\o^2$ and $-\o\,k$ in the second and third equations respectively, the
mean motion is feeding energy from the larger scales to the small scales, namely to $k$, through the $\al_4$-term in the third equation.

To begin with, let us spend a few words on the derivation of the previous model.

\subsection*{Derivation of the model}

As well explained in \cite{Les} (see especially Chapter 4) and \cite{Mo-Pi} (see Chapter 3 therein), for instance, 
the common approach to the turbulence theory starts by decomposing the incompressible velocity $u\,=\,u(t,x)$ 
into its mean part $\lan u\ran$ and fluctuating part $\wtilde u$: namely, in the turbulent regime we write $u\,=\,\lan u\ran+\wtilde{u}$.
Notice that the new velocity fields $\lan u\ran$ and $\wtilde{u}$ are still incompressible.
Injecting the decomposition $u\,=\,\lan u\ran+\wtilde{u}$ into the incompressible Navier-Stokes equations solved by $u$, we derive the so-called
\emph{Reynolds averaged Navier-Stokes equations}, denoted with the acronym RANS.
These are a cascade of differential equations for $\lan u\ran$, $\wtilde u$ and their higher order correlations.
The main problem is that RANS models are not closed.

For instance, one can start by taking the average of the Navier-Stokes system to derive a differential relation for $\lan u\ran$; 
however, owing to the non-linearities coming into play in the original system, when doing that
an additional term appears, often called \emph{Reynolds stress}, which depends on the oscillating component $\wtilde u$. More precisely, one needs an information on
the correlation $\lan \wtilde u\otimes\wtilde u \ran$ in order to close the system and derive an equation for $\lan u\ran$. This reflects the physical principle that small-scale processes affect the mean motion at larger scales.
The problem is that such an information is not given \tsl{a priori}; in order to find it, one can use the equation for $\wtilde u$ and write an equation for the Reynolds stress, but it is easily seen that a third-order correlation appears. In this way, one can generate a cascade
of equations for the correlation functions of any order, but again, the equation for the  $n$-th order correlation will depend on the moments up to order $n+1$.

Thus, a \emph{closure hypothesis} is needed in order to obtain a closed system and work on it. From this point of view, Kolmogorov's system \eqref{eq:kolm_d}
belongs to the class of two-equation models of turbulence, for which, as the name suggests, the closure is given by imposing two additional equations.
After having been almost left aside for a long time in favour of the one-equation models (whose main example is the $k$-model proposed
by Prandtl in \cite{Pra}, and where the closure is given by imposing only one equation, typically for the turbulent kinetic energy $k$),
two-equation models have recently gained a renewed interest, essentially because in those theories the local length scale $\ell$ of turbulence
has not to be prescribed arbitrarily, but can be deduced from the solution of the system.

The most known two-equation models of turbulence are probably the so-called $k-\veps$ models. Here, $\veps$ typically\footnote{Be aware that,
in the original Kolmogorov paper \cite{Kolm_42}, the symbol $\veps$ was used instead to represent the term $|\D u|^2$ appearing in the third equation
of \eqref{eq:kolm_d}.} denotes the turbulent energy dissipation rate. In fact, any pair of variables of the form $k^a\,\veps^b$, with integer values of $a$ and $b$,
can be used to derive a different but equivalent $k-\veps$ model.
In Kolmogorov's system \eqref{eq:kolm_d}, the small-scale unknowns are instead $k$ and $\o$: in this case, the energy dissipation rate $\veps$
and the length scale $\ell$ can be recovered \tsl{via} the formulas
\begin{equation} \label{eq:e-l}
\veps\,=\,k\,\o\qquad\qquad\mbox{ and }\qquad\qquad
\ell\,=\,c\,\frac{\sqrt{k}}{\o}\,,
\end{equation}
where $c>0$ is an adimensional physical constant.

\subsection*{Previous mathematical results}

The systematic mathematical analysis of equations \eqref{eq:kolm_d} is quite recent. Let us give a short account of previous works devoted to their study.

In \cite{Mie-Nau} (see also \cite{Mie-Nau_CRAS} for an announcement of the results), Mielke and Naumann proved the existence of global in time
weak solutions to system \eqref{eq:kolm_d} set on the torus $\T^3$. The solutions constructed therein are analogous in spirit to Leray's finite energy solutions
for the incompressible Navier-Stokes equations, except for the fact that $k$ cannot really belong to $L^2$, because of the last term appearing on the
right-hand side of its equation, which is merely in $L^1$ (yet, the authors are able to prove a gain of integrability for that quantity).
Also, due to the lack of space compactness for that $L^1$ term, the solutions are in fact weak solutions of the system, but one has to admit the presence
of a defect measure in the weak formulation of the equation for $k$. We refer to work \cite{Mie-Nau} for more details.

An important point of the analysis of \cite{Mie-Nau} is that the initial frequency function $\o_0$ and the initial mean turbulent kinetic energy $k_0$
are taken strictly positive: more precisely, the authors assume that
\begin{equation} \label{ass:o-k}
0\,<\,\o_{0,*}\,\leq\,\o_0\,\leq\,\o^*_0\qquad\qquad\mbox{ and }\qquad\qquad k_0\,\geq\,k_{0,*}\,>\,0\,,
\end{equation}
for suitable positive constants $\o_{0,*}$, $\o_0^*$ and $k_{0,*}$. As a matter of fact, those positivity conditions are preserved by the flow, in the sense that,
for later times $t>0$ and for any $x\in\T^3$, one gets
\[
0\,<\,\o_*(t)\,\leq\,\o(t,x)\,\leq\,\o^*(t)\qquad\qquad\mbox{ and }\qquad\qquad k(t,x)\,\geq\,k_*(t)\,>\,0\,,
\]
for suitable time-dependent functions $\o_*(t)$, $\o^*(t)$ and $k_*(t)$. See also our discussion in Subsection \ref{ss:ode} below. Thus,
system \eqref{eq:kolm_d} becomes really parabolic under conditions \eqref{ass:o-k}, and this is a key ingredient in the analysis of \cite{Mie-Nau}.

The same positivity assumption \eqref{ass:o-k} is formulated (and plays a fundamental role)
in works \cite{Kos-Kub} and \cite{Kos-Kub_glob} by Kosewski and Kubica.
There, the authors prove respectively local in time well-posedness, global in time well-posedness for small data, for the Kolmogorov model \eqref{eq:kolm_d}
set in $\T^3$, at $H^2$ level of regularity on the initial data.

At this point, let us comment more on assumption \eqref{ass:o-k}.
As far as our understanding of turbulence theory goes, we observe that, while the first condition on $\o_0$ is somewhat expected from the physical point of view, the strict positivity of $k_0$ may look as an \textsl{ad hoc} assumption, formulated in order to guarantee the non-degeneracy of the parabolic smoothing for the solutions.

To the best of our knowledge, the only work which is able to deal with a possible vanishing of the function $k_0$ is paper \cite{Bul-Mal} by
Bul\'i\v{c}ek and M\'alek. There, the authors study the existence of global in time finite energy (in the same sense as the one described above
for \cite{Mie-Nau}) weak solutions to system \eqref{eq:kolm_d}. Here, the word ``weak'' has to be intended in a classical sense, since
no defect measures are needed. Remarkably, the problem is set on a bounded domain $\Omega\subset\R^3$,
and suitable boundary conditions are formulated for the three unknowns $u$, $k$ (called $b$ in \cite{Bul-Mal}) and $\o$. However,
the authors reformulate the system in terms of the energy function $E$, which is defined (in our notation) as
\[
E\,:=\,\frac{1}{2}\,|u|^2\,+\,\frac{\nu}{\al_4}\,k\,.
\]
The energy function $E$ satisfies a better relation than $k$ itself, inasmuch as the higher order terms appearing in its equation are in a divergence form.
Moreover, for reasons linked again to the lack of integrability for $k$, also the equation for $\o$ is reformulated: essentially,
one integrates by parts the non-linearity occurring in the parabolic term of the second equation of \eqref{eq:kolm_d}. 
Thus, the analysis is performed on a system collecting the (new) equations for $u$, $\o$ and $E$ (and not on the original system),
and the existence of weak solutions is proved for that system.
It is worth noticing that, as pointed out in \cite{Bul-Mal}, the two systems of equations are equivalent for smooth solutions, but they are not
in a weak setting. However, the authors show that, recovering $k$ from $u$ and $E$, the triplet of functions $\big(u,\o,k\big)$ are a kind of
\emph{suitable weak solutions} for the original system \eqref{eq:kolm_d}, where the equation for $\o$ is yet the modified one.

Concerning the possible vanishing of the mean turbulent kinetic energy, in \cite{Bul-Mal} $k_0$ is assumed to satisfy
\[
k_0\,>\,0\,,\qquad\qquad \mbox{ with }\qquad 
\log k_0\,\in\,L^1(\Omega)\,,
\]
instead of \eqref{ass:o-k}. Thus, $k_0$ is not assumed to possess a strictly positive lower bound, but the logarithmic hypothesis
gives a control on its possible vanishing. Correspondingly, the solutions constructed in that paper are proved to satisfy the same type of bounds for $k$,
uniformly in time.

We conclude this part by mentioning a recent result \cite{Mie} by Mielke, who is able to construct weak and very weak solutions
for a class of reduced models of the original Kolmogorov system \eqref{eq:kolm_d} (somehow similar to the toy-model we will discuss below),
allowing for initial data $k_0\geq0$ having general support.

\subsection*{Contents of the paper and main results}
In this paper, we study the $1$-D version of the \emph{Kolmogorov two-equation model of turbulence}.
Dropping the divergence-free condition on the velocity field, and correspondingly erasing the pressure function from the equation for $u$, system \eqref{eq:kolm_d} reduces to
\begin{equation} \label{eq:kolm}
\left\{\begin{array}{l}
       \d_tu\,+\,u\,\d_xu\,-\,\nu\,\d_x\left(\dfrac{k}{\o}\,\d_xu\right)\,=\,0 \\[1ex]
       \d_t\o\,+\,u\,\d_x\o\,-\,\alpha_1\,\d_x\left(\dfrac{k}{\o}\,\d_x\o\right)\,=\,-\,\alpha_2\,\o^2 \\[1ex]
       \d_tk\,+\,u\,\d_xk\,-\,\alpha_3\,\d_x\left(\dfrac{k}{\o}\,\d_xk\right)\,=\,-\,k\,\o\,+\,\alpha_4\,\dfrac{k}{\o}\,\big|\d_xu\big|^2\,,
       \end{array}
\right.
\end{equation}
supplemented with suitable initial data
$$
u_{|t=0}\,=\,u_0\,,\qquad \o_{|t=0}\,=\,\o_0\,,\qquad k_{|t=0}\,=\,k_0\,.
$$
The precise assumptions on those initial data will be specified below. 

We avoid here any consideration about boundary effects, and we set the previous system \eqref{eq:kolm} in a very simple geometry,
namely in the spacial domain
$$
\Omega\,=\,\T\,:=\,[-\pi,\pi]/\sim\,, 
$$
where $\sim$ denotes the equivalence relation which identifies $-\pi$ and $\pi$.
Of course, this is a quite important simplification, because it completely removes all boundary-induced turbulent phenomena in the dynamics.
However, our focus here is rather on understanding the well-posedness \tsl{vs} ill-posedness problem of Kolmogorov's model in the case of vanishing turbulent
kinetic energy $k$.

The main contribution of this paper is twofold. On the one hand, we prove that, even for vanishing $k_0$ (which implies that $k(t)$
also vanishes at any later time),
the system is locally in time well-posed in spaces $H^m$, for any $m\in\N$ such that $m\geq2$. We refer to Theorem \ref{t:local} below for the precise statement.
On the other hand, we show also that there are solutions that blow up in finite time: this is the content of Theorem \ref{t:blow-up}.


It should be noted that
the Kolmogorov system \eqref{eq:kolm} contains several interesting limiting cases.
For instance, setting $\o\equiv0$ and $k\equiv0$, the equation for $u$ reduces to the classical inviscid Burgers equation.
As is well-known, see \tsl{e.g.} \cite{Wh}, for the inviscid Burgers equation there are smooth initial data for which the corresponding solutions
develop a singularity in finite time.
On the other hand, if one takes 
\[
u\equiv0\qquad\qquad \mbox{ and }\qquad\qquad \o=\o(t):=\frac{\oline{\o}}{1+\alpha_2\oline{\o}t}
\]
for some $\oline\o>0$, system \eqref{eq:kolm} reduces to a porus medium equation for $k$ of the form
\begin{equation} \label{intro_eq:PME}
 \d_tk\,-\,\frac{\alpha_3}{2\,\o(t)}\,\dd\big(k^2\big)\,=\,-\,\o(t)\,k\,.
\end{equation}
We refer to the classical book \cite{Va} for an in-depth study of porus medium equations.
For the sake of completeness, we will study equation \eqref{intro_eq:PME} in Appendix \ref{app:PME} and we will establish a global well-posedness result
in $H^1(\T)$ for initial $k_0\geq0$ having general support (although it could be seen, for instance by following the approach of \cite{F-GB_ZAMP},
that there exist smooth solutions in $H^2$ which do blow up in finite time).
All this shows that the structure of equations \eqref{eq:kolm} when $k$ vanishes is rather intricate and mixes non-linear effects, hyperbolic phenomena
and non-trivial coupling between the different quantities. In particular, this can be seen as a (rough, but still consistent) justification
of the assumptions made below for establishing a suitable well-posedness theory.
On the other hand, although both those limiting cases present finite time blow-up of smooth solutions,
the general behaviour of solutions to the Kolmogorov system \eqref{eq:kolm} remains an interesting open question.

More comments about our results will be given below. To begin with,
in order to explain the main difficulties linked to our study and to motivate the assumptions we will formulate in our main theorems,
let us introduce a toy-model. 

\paragraph{A toy-model.} At a first sight, the local in time well-posedness result may appear a trivial consequence of classical quasi-linear
hyperbolic theory applied to system \eqref{eq:kolm}. Actually, this is not completely true,
owing to the possible vanishing of the function $k$, which in fact creates some complications and makes
the argument for the local well-posedness more tricky than that.

In order to see this, let us introduce the following toy-model:
\begin{equation} \label{eq:toy}
\left\{\begin{array}{l}
       \d_tu\,+\,u\,\d_xu\,-\,\d_x\left(\g\,\d_xu\right)\,=\,0 \\[1ex]
       \d_t\g\,+\,u\,\d_x\g\,-\,\d_x\left(\g\,\d_x\g\right)\,=\,\g\,\big|\d_xu\big|^2\,,
       \end{array}
\right.
\end{equation}
supplemented with initial conditions $\big(u,\g\big)_{|t=0}\,=\,\big(u_0,\g_0\big)$, for some $\g_0\geq0$.
Notice that a general class of similar reduced systems has been recently introduced in \cite{Mie}.

In system \eqref{eq:toy}, the function $\g$ essentially plays the role of the viscosity function $k/\o$ appearing in the original system \eqref{eq:kolm}.
Of course, we have retained only the main structural terms of the equation for $k/\o$, namely transport by $u$ and diffusion with diffusion coefficient
equal to $\g$ itself, as well as the apparently most complicated term $\g\,|\d_xu|^2$.
In addition, we have set all the parameters $\nu$ and $\al_j$ equal to $1$.

It is easy to see (for instance, by applying a similar argument to the one presented in Subsection \ref{ss:ode} below) that system \eqref{eq:toy} preserves
the minimal value of $\g$, exactly as equations \eqref{eq:kolm} do for $k$. So, we can deduce that $\g(t)\geq0$ for any time $t\geq0$ for which the solution exists.
Moreover, we can reproduce $L^2$ energy estimates for $u$ and $L^1$ estimates for $\g$, obtaining similar controls as the ones
encoded in the original system.
Actually, the toy-model \eqref{eq:toy} preserves much more properties of the solutions of the original equations,
in particular all the properties which allow us to deduce the singularity formation in finite time, but at this stage we do not need that in our discussion.

The problems come when we try to propagate higher regularity norms of the solution. In order to fix ideas, let us focus on $H^2$ estimates,
which are the simplest ones, because they require to propagate the minimal integer regularity which is expected to be needed for well-posedness.
Now, if we write an equation for $\dd u$, we get
\[
\d_t\dd u\,+\,u\,\d_x\dd u\,-\,\d_x\big(\g\,\d_x\dd u\big)\,=\,\big[u,\dd\big]\d_xu\,+\,\d_x\left(\big[\g,\dd\big]\d_xu\right)\,,
\]
where we have denoted by $[A,B]\,:=\,AB-BA$ the commutator between two operators $A$ and $B$. 
The bad term comes from the second commutator: indeed, when performing energy estimates for $\dd u$, one has to deal with a term of the form
\begin{equation} \label{eq:bad}
\int_\Omega\dd\g\,\d_xu\,\d_x^3u\,dx\,.
\end{equation}
That integral is of course out of control, because any third order derivative can be estimated only if a prefactor $\sqrt\g$ appears in front of it,
but we cannot simply do that because we have no strictly positive lower bound for $\sqrt{\g}$. It is also apparent that the situation does not improve
when integrating by parts, for instance.

However, it is easy to see that the integral \eqref{eq:bad} can be estimated if one disposes of a $L^\infty$ bound on $\d_x\sqrt{\g}$.  Indeed, it is enough to note that
$$
\int_\Omega\dd\g\,\d_xu\,\d_x^3u\,dx\,=\,2\int_\Omega\sqrt{\g}\,\dd\sqrt{\g}\,\d_xu\,\d_x^3u\,dx\,+\,2\int_\Omega(\d_x\sqrt{\g})^2\,\d_xu\,\d_x^3u\,dx\,:
$$
the first term on the right can be immediately bounded by Cauchy-Schwarz inequality, whereas for the second we can integrate by parts once, and then apply Cauchy-Schwarz again.
This consideration motivatives
us to work with the ``good unknown'' $\sqrt{\g}$ rather than with $\g$ itself. Because of Sobolev inequalities, $H^2$ bounds for $\sqrt{\g}$ imply
$H^2$ bounds for $\g$, whereas of course the contrary is not true in general. On the other hand, when reformulating system \eqref{eq:toy} in terms
of $\sqrt{\g}$, the bad terms disappear, because whenever a third order derivative hits $\sqrt{\g}$ or $u$, it will be always possible to couple
that term with another prefactor $\sqrt{\g}$.

To conclude this part, we remark that the regularity condition $\sqrt{\g}\in H^2$ looks quite stringent in order to study
properties linked with the propagation of the support of $\g$. However, our focus here is rather on the well-posedness theory and on the possible
development of singularities in finite time. The argument above shows that, in this context, bypassing the condition $\sqrt{\g}\in H^2$ does not seem
straightforward.

\paragraph{Local well-posedness.}

The previous considerations constitute the main motivation to formulate our first main result, namely well-posedness of equations \eqref{eq:kolm}
in a hyperbolic functional setting, where however conditions will be formulated on $\big(u,\o,\sqrt{k}\big)$ instead of $\big(u,\o,k\big)$.
This is the matter of Theorem \ref{t:local} below.

We consider solutions at $H^m$ level of regularity, for any integer $m\geq2$. Our proof will be based
on elementary methods (see more details below); our goal is to give a simple to understand, but robust argument to construct solutions for vanishing $k$.
Thus, we focus only on the case of regularity of integer order.
On the other hand, as symmetry in $L^2$ of the parabolic part of the equation plays a major role
in our argument (recall the arguments explained for the toy-model), the generalisation to integrability exponents different than $2$ does not really look straightforward.

After those considerations, it is now time to state our first result. Notice that, differently from all previous results, no special assumptions (other than those imposed by the fact that $\sqrt{k_0}$ lies in an appropriate Sobolev space) on the possible
vanishing of $\sqrt{k_0}$ are needed.

\begin{thm} \label{t:local}
Fix $m\in\N$ such that $m\geq 2$.
Let $\big(u_0,\o_0,k_0\big)$ be a triplet of functions belonging to $H^m(\Omega)$. In addition, assume that the following conditions are verified:
\begin{enumerate}[(i)]
 \item there exist two constants such that $0\,<\,\o_*\,\leq\,\o_0\,\leq\,\o^*$;
 \item $k_0\,\geq\,0$ is such that $\bt_0\,:=\,\sqrt{k_0}\,\in\, H^m(\Omega)$.
\end{enumerate}

Then, there exists a time $T>0$ such that system \eqref{eq:kolm}, supplemented with $\big(u_0,\o_0,k_0\big)$ as initial datum,
admits a solution $\big(u,\o,k\big)$ on $[0,T]\times\Omega$ with the following properties:
\begin{enumerate}[(a)]
\item for any time $t\in[0,T]$, one has $k(t)\,\geq\,0$ and $\min_{x\in\Omega}\o(t,x)\,>\,0$;
 \item the functions $u$, $\o$ and $\sqrt{k}$ belong to the space $L^\infty\big([0,T];H^m(\Omega)\big)\,\cap\,\bigcap_{s<m}C\big([0,T];H^s(\Omega)\big)$;
 \item the three functions $\sqrt{k/\o}\,\d_x^{m+1}u$, $\sqrt{k/\o}\,\d_x^{m+1}\o$ and $\sqrt{k/\o}\,\d_x^{m+1}\sqrt{k}$ all belong to the space
$L^2\big([0,T];L^2(\Omega)\big)$.
\end{enumerate}
In addition, such solution is unique in the class
\begin{align*}
\mbb X_T(\Omega)\,&:=\,\Big\{\big(u,\o,k\big)\;\Big|\quad \o,\o^{-1},\,{k}\;\in\,L^\infty\big([0,T]\times\Omega\big)\,,\quad \o>0\,,\quad k\geq0 \\
&\qquad\qquad\qquad\quad
u,\,\o,\,\sqrt{k}\;\in\,C\big([0,T];L^2(\Omega)\big)\,,\quad \d_xu,\,\d_x\o,\,\d_x\sqrt{k}\;\in\,L^\infty\big([0,T]\times\Omega\big) \Big\}\,.
\end{align*}
Finally, denote by $T^*>0$ the lifespan of the previous solution, and define
\[
\mc E_0\,:=\,\left\|u_0\right\|_{H^2}^2\,+\,\left\|\o_0\right\|_{H^2}^2\,+\,\left\|\sqrt{k_0}\right\|_{H^2}^2\,.
\]
Then there exists a constant $C>0$, depending only on $(\nu,\al_1\ldots\al_4)$, such that one has
\begin{equation} \label{est:lower-T}
T^*\,\geq\,\min\left\{1\,,\,\frac{C}{\left(\max\big\{1,\mc E_0\big\}\right)^2}\right\}\,.
\end{equation}
Moreover, 
one has $T^*\,<\,+\infty$ if and only if
\begin{equation} \label{cont-cont}
\int^{T^*}_0\left(1+\left\|\sqrt{k}\right\|^2_{L^\infty}\right)\,\left\|\left(\d_xu,\d_x\o,\d_x\sqrt{k}\right)\right\|^2_{L^\infty}\,dt\,=\,+\infty\,.
\end{equation}
\end{thm}

A remark on the previous theorem is in order.
\begin{rmk} \label{r:reg}
Notice that the continuation criterion \eqref{cont-cont} implies, as is classical in quasi-linear hyperbolic problems, that the lifespan of the
solution does not depend on its regularity: if the datum is smooth, say in $H^m$ with $m>2$ large, the lifespan of the corresponding solution as a
$H^2$ solution is the same as a $H^m$ solution.

In particular, this justifies the lower bound \eqref{est:lower-T} in terms only of the $H^2$ norm of the initial datum.
\end{rmk}
 
The proof of Theorem \ref{t:local} is based on elementary arguments. We approximate problem \eqref{eq:kolm} with vanishing initial $k_0$ with problems
having initial $k_{0,\veps}$'s which satisfy $k_{0,\veps}\geq\veps$. We thus construct approximate solutions by using the result from \cite{Kos-Kub},
and we pass to the limit in the approximation parameter $\veps\ra0^+$ by using a compactness argument.
This proves the existence, for which we refer to Subsection \ref{ss:local}. The relevant uniform bounds in suitable norms will be shown before,
in Section \ref{s:a-priori}.

Uniqueness of solutions is based on a simple stability estimate in $L^2$: see Subsection \ref{ss:unique}. Notice that our argument cannot be adapted
to prove a weak-strong uniqueness principle for system \eqref{eq:kolm}, since higher regularity is required on both families of solutions.
We refer to Remark \ref{r:weak-strong} for further comments about this.

To conclude this part, we point out that, in order to lighten the presentation, our proof will be limited to treat the lower regularity case only, namely
the case $m=2$. Generalisation to the case of $m>2$ is rather straightforward, hence omitted.

\paragraph{Singularity formation.}

Next, we show that the solutions we previously constructed in Theorem \ref{t:local} actually blow up in finite time, in general. This is the content
of the next result. 
That statement will be made more precise in Subsection \ref{ss:blow-up}, by formulating Theorem \ref{t:symmetry}.

\begin{thm} \label{t:blow-up}
There exist smooth initial data $\big(u_0,\o_0,k_0\big)$, such that the corresponding solutions $\big(u,\o,k\big)$ to system \eqref{eq:kolm} blow up in finite time.
\end{thm}

The mechanism behind our proof of the blow-up is similar in spirit to the one for the Burgers equation.
In particular, our proof deeply relies on the compressibility of the velocity field $u$. Compared to the Burgers case, here it is fundamental that the ``viscosity
coefficient '' $k/\o$ vanishes, hence that $k$ itself vanishes; nonetheless, the basic issue here is that $k$ may vanish only at one point, and this is enough to create
a singularity. Of course, this is very much in contrast with the case where $k\geq \delta>0$, for which global existence of smooth solutions holds. 

Before moving on, let us spend a few more words to comment Theorem \ref{t:blow-up}. It is tempting to interpret the finite time singularity formation
as an indication of the fact that the Kolmogorov model \eqref{eq:kolm_d} is not a good model for describing turbulence. This remark would be
pretty much in connection with some physical theories, which criticise $k-\veps$ models and their pertinence for the theoretical understanding of
turbulence in fluids.

However, we have to point out that system \eqref{eq:kolm} is merely a $1$-D reduction of the original Kolmogorov model \eqref{eq:kolm_d}, and,
as such, it loses very much physics with respect to that system. In particular, one may object that real turbulence is expected to develop
only in higher space dimensions. 
Also, in system \eqref{eq:kolm} we neglect the divergence-free condition
on the velocity field, as well as the pressure term. This simplification kills completely non-local effects (often called long-range interactions
in physical theories), which though are strongly supposed to play an important role in turbulence (see \tsl{e.g.} Chapter 7 of \cite{Les}
and Chapter 16 of \cite{CR-B}).

More realistically, and as far as our understanding of turbulence goes, we can interpret Theorem \ref{t:blow-up} as a reflection of the fact
that the model is formulated by Kolmogorov in order to describe a fully developed turbulence, so it loses its meaning
in the regime of vanishing turbulent kinetic energy.
As a matter of fact, in Kolmogorov's theory of fully developed turbulence, the turbulent kinetic energy produced at small scales, whose typical length scale $\ell$
is defined in \eqref{eq:e-l}, is completely dissipated at large scales by viscosity. However, for vanishing $k\sim0$ the definition of the typical length
scale $\ell$ degenerates. 
Therefore, we expect that, for vanishing $k$, one should rather change the model.

\subsubsection*{Organisation of the paper}

Let us present now a short overview of the paper.
First of all, in Section \ref{s:LP} we collect some useful tools from Fourier analysis which are needed in our study.
In Section \ref{s:a-priori} we present \tsl{a priori} estimates for system \eqref{eq:kolm}. They will be the key step in the proof of our results,
namely Theorems \ref{t:local} and \ref{t:blow-up}, which will be presented in Section \ref{s:proof}.

\section{Littlewood-Paley theory on the torus} \label{s:LP}
We recall here a few elements of Littlewood-Paley theory and use them to derive some useful inequalities.
We refer \tsl{e.g.} to Chapter 2 of \cite{B-C-D} for details on the construction in the $\R^d$ setting, to reference \cite{Danchin} for the adaptation to the
case of a $d$-dimensional periodic box $\T^d_a$, where $a\in\R^d$
(this means that the domain is periodic in space with, for any $1\leq j\leq d$, period equal to $2\pi a_j$ with respect to the $j$-th component).

For simplicity of presentation, we focus here on the one-dimensional case $d=1$, the only relevant one for our study, and on the case where $a_1=1$.
Hence, we simply write the spacial domain as $\Omega\,=\,\T$.
We also denote by
$|\T|\,=\,\mc L(\T)$ the Lebesgue measure of the box $\T$.

\medbreak
First of all, let us introduce some notation and basic material. For a given function $u\in L^1(\T)$, we denote by $\mc Fu\,=\,\big(\what u_k\big)_{k\in\Z}$
its Fourier series, where the Fourier coefficients $\what u_k$ are defined by
\[
\what u_k\,:=\,\frac{1}{|\T|}\int_\T u(x)\,e^{-ik\cdot x}\,dx\,.
\]
Recall that, whenever $1<p<+\infty$, the Carleson-Hunt theorem yields the pointwise (a.e. on $\T$) convergence
of the Fourier series of $u$ to the function $u$ itself:
\begin{equation} \label{eq:Four-ser}
u(x)\,=\,\sum_{k\in\Z}\what{u}_k\,e^{ik\cdot x}\,. 
\end{equation}
We refer to \tsl{e.g.} Chapter 1 of \cite{Duoan} and Chapter 10 of \cite{Edw} for further details. By duality, we can then extend the validity
of formula \eqref{eq:Four-ser} to any distribution $u\in\mc D'(\T)$, where, following Triebel \cite{Triebel} (see Section 9.1 therein),
we have defined $\mc D'(\T)$ to be the topological dual of the set $\mc D(\T)\,:=\, C^\infty(\T)$
of smooth functions on the torus.

Next, we introduce the so called \textit{Littlewood-Paley decomposition}, based on a non-homogeneous dyadic partition of unity with
respect to the Fourier variable. 
We fix a smooth scalar function $\vphi$ such that $0\leq \vphi\leq 1$, $\vphi$ is even and supported in the ring $\left\{r\in\R\,\big|\ 5/6\leq |r|\leq 12/5 \right\}$, and such that
\[
\forall\;r\in\R\setminus\{0\}\,,\qquad\qquad \sum_{j\in\Z}\vphi\big(2^{-j}\,r\big)\,=\,1\,.
\]

Let us define $|D|\,:=\,(-\dd)^{1/2}$ as the Fourier multiplier\footnote{Throughout we agree  that  $f(D)$ stands for 
the pseudo-differential operator $u\mapsto\mc{F}^{-1}(f\,\mc{F}u)$.} of symbol $|k|$, for $k\in\Z$.
The dyadic blocks $(\Delta_j)_{j\in\Z}$ are then defined by
$$
\forall\;j\in\Z\,,\qquad\qquad
\Delta_ju\,:=\,\varphi(2^{-j}|D|)u\,=\,\sum_{k\in\Z}\vphi(2^{-j}|k|)\,\what u_k\,e^{ik\cdot x}\,.
$$
Notice that, for $j<0$ negative enough (in general, depending on the box $\T^d_a$), one has $\Delta_j\equiv0$. In addition, one has the following Littlewood-Paley decomposition in $\mc D'(\T)$ (see details in Section 9.1 of \cite{Triebel}):
$$
\forall\;u\in\mc{D}'(\T)\,,\qquad\qquad u\,=\,\what u_0\,+\, \sum_{j\in\Z}\Delta_ju\qquad\mbox{ in }\quad \mc D'(\T)\,.
$$
Recall that $\what u_0$ is simply the mean value of $u$ on $\T$:
\[
\what u_0\,=\,\oline u\,=\,\,\frac{1}{|\T|}\int_\T u(x)\,dx\,.
\]


We explicitly remark that, for any $j\in\Z$, the Fourier multipliers $\Delta_j$ are linear operators which are bounded on $L^p$ for any $p\in[1,+\infty]$, with norm \emph{independent} of $j$
and $p$.

Now, we present a simplified version of the classical \emph{Bernstein inequalities}, which turns out to be enough for our scopes. We refer to Chapter 2 of \cite{B-C-D}
for the statement in its full generality.

\begin{lemma} \label{l:bern}
There exists a universal constant $C>0$ depending only on the size of the torus $\T$ and on the support of the function $\vphi$ fixed above,
such that the following properties hold true:
for any $j\in\Z$,
for any $m\in\N$, for any couple $(p,q)\in[1,+\infty]^2$ such that $p\leq q$,  and for any $u\in \mc D'(\T)$,  we  have
\begin{align*}
&\left\|\Delta_ju\right\|_{L^q}\, \leq\,
 C\,2^{j\left(\frac{1}{p}-\frac{1}{q}\right)}\,\|\Delta_ju\|_{L^p} \\[1ex]
&\qquad\qquad\qquad\mbox{ and }\qquad\qquad
C^{-m-1}\,2^{-jm}\,\|\Delta_ju\|_{L^p}\,\leq\,
\|\d_x^m \Delta_ju\|_{L^p}\,\leq\,C^{m+1} \, 2^{jm}\,\|\Delta_ju\|_{L^p}\,.
\end{align*}
\end{lemma}   


It is well known that, for all $s\in\R$, the space $H^s$ can be characterised in terms of Littlewood-Paley decomposition (see Section 2.7 of \cite{B-C-D}). In particular,
one has the equivalence
\begin{equation} \label{eq:LP-Sob}
\|u\|^2_{H^s}\,\sim\,\left|\what u_0\right|^2\,+\,\sum_{j\in\Z}2^{2sj}\,\|\Delta_ju\|_{L^2}^2\,.
\end{equation}



Next, we recall some well-known interpolation inequalities. Remark that different proofs with respect to the classical ones
may be given by use of the Littlewood-Paley decomposition.
We start by recalling a classical interpolation result in Sobolev spaces:
\[ 
\forall\,f\in H^2(\T)\,, 
\qquad\qquad
\left\|\d_xf\right\|_{L^2}\,\lesssim\,\left\|f\right\|_{L^2}^{1/2}\;\left\|\dd f\right\|_{L^2}^{1/2}\,.
\] 
Moreover, we have the following special case of Gagliardo-Nirenberg inequality:
\begin{equation} \label{est:interpolation}
\forall\,f\in H^2(\T)\,, 
\qquad\qquad
\left\|\d_x f\right\|_{L^4}\,\lesssim\,\left\|f\right\|_{L^\infty}^{1/2}\;\left\|\dd f\right\|_{L^2}^{1/2}\,.
\end{equation}
Similarly, the following Gagliardo-Nirenberg type inequality also holds true:
\begin{equation} \label{est:interp_2}
 \forall\,f\in H^1(\T)\quad \mbox{ such that }\;\oline f\,=\,0\,,\qquad\qquad
\left\|f\right\|_{L^\infty}\,\lesssim\,\left\|f\right\|_{L^p}^{p/(p+2)}\;\left\|\d_xf\right\|_{L^2}^{2/(p+2)}\,,
\end{equation}
for any $p\in[1,+\infty]$.

\section{\textsl{A priori} estimates} \label{s:a-priori}

In this section, we collect useful \tsl{a priori} estimates satisfied by smooth solutions $(u,\o,k)$ to system \eqref{eq:kolm}, defined in $\R_+\times\Omega$,
where we recall that $\Omega\,=\,\T$. We will see in Section \ref{s:proof}, and more precisely in Subsection \ref{ss:local} how to use those bounds in order to prove local well-posedness at the level of regularity
claimed in Theorem \ref{t:blow-up}.

In all the estimates below, we have to put special attention on the dependence of all the estimates on the lower bound of $k$, which may vanish, under our assumptions.

\subsection{ODE analysis} \label{ss:ode}
We start by deriving pointwise upper and lower bounds for the functions $\o$ and $k$.
Those bounds have already been derived in \cite{Mie-Nau_CRAS}, \cite{Mie-Nau} by a comparison principle for weak solutions
to scalar parabolic equations. Here we propose a different method, based on an explicit ODE analysis.
Recall that, throughout this section, the solution $(u,\o,k)$ is assumed to be smooth, defined on the whole $\R_+\times\Omega$.

First of all, let us define $\lambda$ and $\mu$ as the solutions to the following system of ODEs:
\begin{equation} \label{eq:ODE}
\left\{\begin{array}{ll}
       \dfrac{d}{dt}\lambda\,=\,-\,\alpha_2\,\lambda^2\,, &\quad \lambda_{|t=0}\,=\,\lambda_0 \\[2ex]
       \dfrac{d}{dt}\mu\,=\,-\,\lambda\,\mu\,, &\quad \mu_{|t=0}\,=\,\mu_0\,,
       \end{array}
\right.
\end{equation}
for some $\lambda_0$ and $\mu_0$ to be fixed later. System \eqref{eq:ODE} can be solved explicitly: we have
$$
\lambda(t)\,=\,\frac{\lambda_0}{\lam_0\,\alpha_2\,t\,+\,1}\qquad\mbox{ and }\qquad \mu(t)\,=\,\frac{\mu_0}{\left(\lam_0\,\alpha_2\,t\,+\,1\right)^{1/\alpha_2}}\,.
$$

Next,  let us introduce the quantities 
$$
\omega^{*}\,:=\,\max_{x\in\Omega}\o_0(x)\,,\qquad \omega_{*}\,:=\,\min_{x\in\Omega}\o_0(x)\,,\qquad k_*\,:=\,\min_{x\in\Omega}k_0(x)
$$
as the initial upper and lower bounds of the functions $\o$ and $k$. We also define their counterparts at any time $t\geq0$:
$$
\o_{\min}(t)\,:=\,\min_{x\in\Omega}\o(t,x),\qquad\qquad \o_{\max}(t)\,:=\,\max_{x\in\Omega}\o(t,x) \qquad\mbox{ and }\qquad k_{\min}(t)\,:=\,\min_{x\in\Omega}k(t,x)\,.
$$
Due to the smoothness of the solutions, these functions are Lipschitz and, as a consequence of Rademacher theorem, they are differentiable almost everywhere. 

Denoting by $x_{\min}=x_{\min}(t)$ the point where the minimum value of $k(t)$ is attained, we obtain that $k_{\min}(t)$ satisfies the ODE
\begin{equation} \label{eq:d_tk}
\frac{d}{dt}k_{\min}\,=\,\alpha_3\left(\frac{k}{\o}\,\d_x^2k\right)_{\!\!|x=x_{\min}}\,-\,k_{\min}\,\o_{|x=x_{\min}}\,+\,
\alpha_4\left(\frac{k}{\o}\,\left|\d_xu\right|^2\right)_{\!\!|x=x_{\min}}\,.
\end{equation}
Using Gr\"onwall inequality we find that
$$
k_{\min}(t)\geq0.
$$

With the same ideas, and denoting by $y_{\min}=y_{\min}(t)$ the point where $\o(t)$ attains its minimum value, we find that $\o_{\min}$ satisfies the ODE
\[
\frac{d}{dt}\o_{\min}\,=\,\alpha_1\,\frac{k_{|x=y_{\min}}}{\o_{\min}}\,\left(\d_x^2\o\right)_{|x=y_{\min}}\,-\,\alpha_2\,\o_{\min}^2\,.
\]
We could see that $\o_{\min}$ remains positive by \tsl{e.g.} a continuity argument. However, we decide here to perform explicit computations and get
\[
 \frac{1}{2}\,\frac{d}{dt}\left(\o_{\min}^2\right)\,=\,\alpha_1\,k_{|x=y_{\min}}\,\left(\d_x^2\o\right)_{|x=y_{\min}}\,-\,\alpha_2\,\o_{\min}^3\,\geq\,
-\,\alpha_2\,\o_{\min}^3\,.
\]
Simple computations immediately imply that
\[ 
\o_{\min}(t)\,\geq\,\frac{\o_*}{\o_*\,\alpha_2\,t\,+\,1}\qquad\qquad \forall\;t\geq0\,.
\] 

Next, we observe that, at the point $y_{\max}=y_{\max}(t)$ which realises the maximum value of $\o(t)$, one has
$$
\frac{d}{dt}\o_{\max}\,=\,\alpha_1\left(\frac{k}{\o}\,\d_x^2\o\right)_{|x=y_{\max}}\,-\,\alpha_2\,\omega_{\max}^2\,\leq\,-\,\alpha_2\,\omega_{\max}^2\,.
$$
Hence, solving the first equation in \eqref{eq:ODE} with initial datum $\lambda_0=\omega^*$, we deduce
\[ 
\o_{\max}(t)\,\leq\,\frac{\o^*}{\o^*\,\alpha_2\,t\,+\,1}\qquad\qquad \forall\;t\geq0\,.
\] 

In the end, we get the following pointwise estimate for $\o$:
\begin{equation} \label{est:omega_L^inf}
\forall\;(t,x)\,\in\,\R_+\times\Omega\,,\quad
0\,<\,\frac{\o_*}{\o_*\,\alpha_2\,t\,+\,1}\,\leq\,\o_{\min}(t)\,\leq\,\omega(t,x)\,\leq\,\o_{\max}(t)\,\leq\,\frac{\o^*}{\o^*\,\alpha_2\,t\,+\,1}\,.
\end{equation}

Using the previous bounds \eqref{est:omega_L^inf} in equation \eqref{eq:d_tk}, in turn we find
\[
 \frac{d}{dt}k_{\min}\,\geq\,
-\,k_{\min}\,\o_{\max}\,.
\]
Solving the second equation \eqref{eq:ODE}, where we take $\lambda(t)=\o_{\max}(t)$ and $\mu_0=k_*$, we infer
\begin{equation} \label{est:k_lower}
k_{\min}(t)\,\geq\,\frac{k_*}{\left(\o^*\,\alpha_2\,t\,+\,1\right)^{\!1/\alpha_2}}\qquad\qquad \forall\;t\geq0\,.
\end{equation}
Notice that, under the assumptions of Theorem \ref{t:local}, we have $k_*=0$. Therefore, we deduce that
\begin{equation} \label{est:k-inf}
\forall\;(t,x)\in\R_+\times\Omega\,,\qquad\qquad k(t,x)\,\geq\,0\,.
\end{equation}


\subsection{Basic energy estimates} \label{ss:energy}
We now perform basic $L^2$ estimates for the smooth solution $(u,\o,k)$, defined in $\R_+\times\Omega$. Those estimates will provide us with a control on the low frequencies of the solution. The bounds for the high frequencies will be given in the next subsection.

To begin with, we test the equation for $u$ against $u$ itself, and then integrate over $\Omega$: after performing suitable integrations by parts, we find
$$
\frac{1}{2}\,\frac{d}{dt}\int_\Omega|u|^2\,dx\,+\,\nu\int_\Omega\frac{k}{\o}\,\big|\d_xu\big|^2\,dx\,=\,0\,.
$$
After an integration in time, 
we immediately deduce the following estimate: for all $t\geq 0$, one has
\begin{equation} \label{est:u}
\left\|u(t)\right\|_{L^2}^2\,+\,2\,\nu\int^t_0\int_\Omega\frac{k}{\o}\,\big|\d_xu\big|^2\,dx\,d\tau\,\leq\,\|u_0\|^2_{L^2}\,.
\end{equation}

Repeating the same computations for the second equation in \eqref{eq:kolm}, tested this time against the function $\o$, we get
\begin{equation} \label{est:d-dt_o}
\frac{1}{2}\,\frac{d}{dt}\int_\Omega|\o|^2\,dx\,+\,\alpha_1\int_\Omega\frac{k}{\o}\,\big|\d_x\o\big|^2\,dx\,+\,\alpha_2\int_\Omega\o^3\,dx\,=\,\frac{1}{2}\int_\Omega\d_xu\,\o^2\,.
\end{equation}
We have to control the last term in the right-hand side of the previous equality. There are several possibilities for that; for instance, the cheapest choice in terms of regularity consists in taking advantage of the last term on the left-hand side of \eqref{est:d-dt_o}, thus making\footnote{From now on, we will use the following notation: if $X$ is a Fr\'echet space of functions over $\Omega$, for any finite $T > 0$ we note $L^p_T(X) = L^p\big([0, T] ; X\big)$ and $C_T(X) = C\big([0, T];X\big)$.} a $L^3_T(L^3)$ norm of $\d_xu$ appear. In turn, this would be fine, in view of the higher order estimates of Subsection \ref{ss:higher} below. However,
here we opt for a simpler control and employ a $L^\infty$ bound for $\d_xu$, which will be justified from the higher order estimates of Subsection \ref{ss:higher}:
we get
\begin{align*}
\left|\int_\Omega\d_xu\,\o^2\right|\,&\leq\,
\left\|\d_xu\right\|_{L^\infty}\,\|\o\|^2_{L^2}\,.
\end{align*}
Plugging this bound into \eqref{est:d-dt_o} and integrating in time, by Gr\"onwall's lemma we get the following estimate: for any time $t\geq0$, one has
\begin{align} \label{est:omega}
\left\|\o(t)\right\|_{L^2}^2\,+\int^t_0\int_\Omega\frac{k}{\o}\,\big|\d_x\o\big|^2\,dx\,d\tau\,+\,\left\|\o\right\|^3_{L^3_t(L^3)}\,\leq\,\|\o_0\|^2_{L^2}
\,\exp\left(C\int^t_0\left\|\d_xu(\t)\right\|_{L^\infty}\,d\t\right)\,,
\end{align}
where the multiplicative constant $C>0$ depends only on the values of $\al_1$ and $\al_2$.
Recall that we dispose also of $L^\infty$ controls on $\o$, which are provided by \eqref{est:omega_L^inf}.


\medbreak
Let us turn our attention to the estimates for $k$. 
Looking at the third equation of \eqref{eq:kolm}, it is easy to realise that a $L^2$ estimate for $k$ has no chance to be closed, owing to the presence of a right-hand side which is merely $L^1$.
Therefore, we rather perform an $L^1$ estimate on $k$, obtaining, after an integration by parts, the relation
\begin{equation} \label{est:d-dt_k}
\frac{d}{dt}\int_\Omega k\,dx\,+\,\int_\Omega k\,\o\,dx\,=\,\alpha_4\int_\Omega\frac{k}{\o}\,\big|\d_xu\big|^2\,dx\,+\,\int_\Omega k\,\d_xu\,dx\,.
\end{equation}
Notice that, following the approach of \cite{Mie-Nau}, it is possible to propagate better integrability for $k$ and its derivative. However, the bounds coming from the previous relation \eqref{est:d-dt_k} are enough for our scopes. Indeed, we can write
\[
\left|\int_\Omega k\,\d_xu\,dx\right|\,=\,\left|\int_\Omega\sqrt{k\,\o}\,\sqrt{\frac{k}{\o}}\,\d_xu\,dx\right|\,\leq\,\frac{1}{2}\int_\Omega k\,\o\,dx\,+\,\frac{1}{2}\int_\Omega\frac{k}{\o}\,\big|\d_xu\big|^2\,dx\,.
\]
Therefore, by use of \eqref{est:u} and noticing that the constant may depend on the parameters $\alpha_i$ and $\nu$, from \eqref{est:d-dt_k} we deduce, after an integration in time, that for any time $t\geq0$, one has
\begin{equation} \label{est:k}
\left\|k(t)\right\|_{L^1}\,+\,\int^t_0\int_\Omega k\,\o\,dx\,d\tau\,\lesssim\,\|k_0\|_{L^1}\,+\,\|u_0\|^2_{L^2}\,,
\end{equation}
for an implicit multiplicative constant, which depends only on the parameters $\nu$ and $\alpha_4$. Observe that, owing to our assumption
on $k_0$, we have
\[ 
\|k_0\|_{L^1}\,=\,\left\|\sqrt{k_0}\right\|_{L^2}^2\,<\,+\,\infty\,.
\] 

As it will appear clear from the computation in the next subsection, we also need some control on how the function $k$ touches zero.
This piece of information is obtained by propagating regularity for $\sqrt{k}$, rather than for $k$ itself.
For notational convenience, let us introduce the quantities 
\[
\beta\,:=\,\sqrt{k}\qquad\qquad\mbox{ and }\qquad\qquad \bt_0\,:=\,\sqrt{k_0}\,.
\]

To begin with, we use again the fact that, for any $t\geq0$, one has $\left\|\bt(t)\right\|^2_{L^2}\,=\,\left\|k(t)\right\|_{L^1}$.
Therefore, owing to estimate \eqref{est:k}, we deduce, for all $t\geq0$, the bound
\begin{equation} \label{est:beta_L^2}
 \left\|\bt(t)\right\|_{L^2}\,\lesssim\,\left\|\bt_0\right\|_{L^2}\,+\,\left\|u_0\right\|_{L^2}\,.
\end{equation}

Next, in view of the computations of Subsection \ref{ss:higher} below, we derive an equation for $\beta$. For this,
we start by multiplying the third equation in \eqref{eq:kolm} by the factor $1/\left(2\,\sqrt{k}\right)$; then, straightforward manipulations yield
\begin{equation} \label{eq:beta}
\d_t\bt\,+\,u\,\d_x\bt\,-\,\al_3\,\d_x\left(\frac{\bt^2}{\o}\,\d_x\bt\right)\,+\,\frac{1}{2}\,\bt\,\o\,=\,
\frac{\alpha_4}{2}\,\frac{\bt}{\o}\,\left|\d_xu\right|^2\,+\,\alpha_3\,\frac{\bt}{\o}\,\left|\d_x\bt\right|^2\,.
\end{equation}
To conclude this part, we also recast the first and second equations of \eqref{eq:kolm} in terms of the new set of unknowns $\big(u,\o,\bt\big)$.
We get the new system
\begin{equation} \label{eq:kolm_2}
\left\{\begin{array}{l}
       \d_tu\,+\,u\,\d_xu\,-\,\nu\,\d_x\left(\dfrac{\bt^2}{\o}\,\d_xu\right)\,=\,0 \\[1ex]
       \d_t\o\,+\,u\,\d_x\o\,-\,\alpha_1\,\d_x\left(\dfrac{\bt^2}{\o}\,\d_x\o\right)\,=\,-\,\alpha_2\,\o^2 \\[1ex]
        \d_t\bt\,+\,u\,\d_x\bt\,-\,\al_3\,\d_x\left(\dfrac{\bt^2}{\o}\,\d_x\bt\right)\,+\,\dfrac{1}{2}\,\bt\,\o\,=\,
\dfrac{\alpha_4}{2}\,\dfrac{\bt}{\o}\,\left|\d_xu\right|^2\,+\,\alpha_3\,\dfrac{\bt}{\o}\,\left|\d_x\bt\right|^2\,.
       \end{array}
\right.
\end{equation}
The higher order estimates of the next subsection will be performed on the previous system \eqref{eq:kolm_2}. As a matter of fact,
it is much more natural to work with the unknown $\bt\,=\,\sqrt{k}$ rather than with $k$ itself.

\medbreak
Before moving on, we notice that the basic $L^2$ parabolic estimates leading to a smoothing effect for the solution are not closed yet, because one still needs to control the last term appearing
in the right-hand side of \eqref{est:omega}. This is a purely $1$-D effect, since now it is no more meaningful to assume the velocity field $u$ to be
divergence-free. 

Of course, those estimates will be closed by using higher order bounds on $\big(u,\o,\bt\big)$, which are the matter of the next subsection.

\subsection{Higher order estimates} \label{ss:higher}

Here, we perform higher order Sobolev estimates for the smooth solution $(u,\o,\bt)$, defined on $\R_+\times\Omega$. In the end, we will need to close the estimates in some (possibly small) time interval $[0,T]$. 
Let us recall that we limit ourselves to treat the case of solutions at $H^2$ level of regularity; it goes without saying that the same method applies also to give
propagation of $H^m$ norms, for any $m\geq 2$.
We expect an extension to almost critical $H^s$ regularities, for any $s>3/2$ (or even to $B^{3/2}_{2,1}$ regularity), to be also possible, by use of Littlewood-Paley analysis and paradifferential calculus techniques. However,
our goal here is to give a simple proof by elementary methods, and to put in evidence in a clear way the special structure of the system, for which the new quantity
$\bt\,=\,\sqrt{k}$ plays a key role.

Finally, we stress that, once again, particular care has to be used while dealing with the function $k$, owing to possible regions where this quantity vanishes.

\medbreak
Before starting the estimates, let us observe that, for any $p\in[1,+\infty]$ and any $f\in L^p(\Omega)$, we have
\[
\left|\what f_0\right|\,\lesssim\,\|f\|_{L^1}\,\lesssim\,\|f\|_{L^p}\,.
\]
Thus, owing to the Bernstein inequalities of Lemma \ref{l:bern} and the Littlewood-Paley characterisation \eqref{eq:LP-Sob} of Sobolev spaces, we have
\begin{equation*}
\|u\|_{H^2}^2\,\sim\,\|u\|_{L^2}^2\,+\,\left\|\d^2_{x}u\right\|_{L^2}^2\,, 
\end{equation*}
and analogous equivalences for $\o$ and $\bt$.
Therefore, we only need to bound the second order derivatives of the solution in $L^2$. We will do this in an elementary way, by performing simple energy estimates on the equations satisfied by $\dd u$,
$\dd \o$ and $\dd \bt$.

\medbreak
We start by writing an equation for $\dd u$. Differentiating the first equation in \eqref{eq:kolm_2} twice with respect to the space variable, we find
\begin{equation} \label{eq:ddu}
\d_t\dd u\,+\,u\,\d_x\dd u\,-\,\nu\,\d_x\left(\frac{\bt^2}{\o}\,\d_x^3u\right)\,=\,f\,,
\end{equation}
where we have defined
\[
f\,=\,\sum_{j=1}^4f_j:=\,-\,3\,\d_xu\,\dd u\,+\,3\,\nu\,\dd\left(\frac{\bt^2}{\o}\right)\,\dd u\,+\,2\,\nu\,\d_x\left(\frac{\bt^2}{\o}\right)\,\d_x^3u\,+\,\nu\,\d_x^3\left(\frac{\bt^2}{\o}\right)\,\d_xu\,.
\]
It is not hard to see that the last term appearing in the definition of $f$ is the most dangerous one, when performing energy estimates on $\dd u$.

Likewise, we get an equation for $\dd\o$:
\begin{equation} \label{eq:dd-omega}
\d_t\dd\o\,+\,u\,\d_x\dd\o\,-\,\al_1\,\d_x\left(\frac{\bt^2}{\o}\,\d_x^3\o\right)\,+\,2\,\al_2\,\o\,\dd\o\,=\,g\,,
\end{equation}
where this time the external force $g$ is defined by
\begin{align*}
g\,=\,\sum_{j=1}^6g_j\,&:=\,-\,2\,\d_xu\,\dd\o\,-\,\dd u\,\d_x\o \\
&\qquad +\,3\,\al_1\,\dd\left(\frac{\bt^2}{\o}\right)\,\dd\o\,+\,2\,\al_1\,\d_x\left(\frac{\bt^2}{\o}\right)\,\d_x^3\o\,+\,\al_1\,\d_x^3\left(\frac{\bt^2}{\o}\right)\,\d_x\o\,-\,2\,\al_2\,\left(\d_x\o\right)^2\,.
\end{align*}

 Finally, we repeat the same computations for $\beta$, finding
\begin{equation} \label{eq:dd-k}
\d_t\dd \bt\,+\,u\,\d_x\dd \bt\,-\,\al_3\,\d_x\left(\frac{\bt^2}{\o}\,\d_x^3\bt\right)\,+\,\frac{1}{2}\,\o\,\dd \bt\,=\,h\,,
\end{equation}
where we have denoted by $h$ the quantity
\begin{align*}
h\,=\,\sum_{j=1}^{15}h_j\,&:=\,-\,2\,\d_xu\,\dd \bt\,-\,\dd u\,\d_x\bt\,-\,\d_x\o\,\d_x\bt\,-\,\dfrac{1}{2}\,\bt\,\dd\o \\
&\quad +\,3\,\al_3\,\dd\left(\frac{\bt^2}{\o}\right)\,\dd \bt\,+\,2\,\al_3\,\d_x\left(\frac{\bt^2}{\o}\right)\,\d_x^3\bt\,+\,\al_3\,\d_x^3\left(\frac{\bt^2}{\o}\right)\,\d_x\bt \\
&\quad +\,\dfrac{\al_4}{2}\,\dd\left(\frac{\bt}{\o}\right)\,\big|\d_xu\big|^2\,+\,2\,\al_4\,\d_x\left(\frac{\bt}{\o}\right)\,\d_xu\,\dd u\,+\,\al_4\,\frac{\bt}{\o}\,\left(\dd u\right)^2\,+\,\al_4\,\frac{\bt}{\o}\,\d_xu\,\d_x^3u \\
&\quad +\,\al_3\dd\left(\frac{\bt}{\o}\right)\,\big|\d_x\bt\big|^2\,+\,2\,\al_4\d_x\left(\frac{\bt}{\o}\right)\,\d_x\bt\,\dd\bt\,+\,2\,\al_4\frac{\bt}{\o}\,\left(\dd\bt\right)^2\,+\,2\,\al_4\frac{\bt}{\o}\,\d_x\bt\,\d_x^3\bt\,.
\end{align*}

In the next paragraphs, we will perform $L^2$ estimates in \eqref{eq:ddu}, \eqref{eq:dd-omega} and \eqref{eq:dd-k}. Of course, the key point will be the control of the forcing terms $f$, $g$ and $h$
in the energy space. 

\subsubsection{The $H^2$ estimate for $u$} \label{sss:u}

We start by testing equation \eqref{eq:ddu}: simple computations give us
\begin{equation} \label{est:dd-u_part}
 \frac{1}{2}\,\frac{d}{dt}\left\|\dd u\right\|_{L^2}^2\,+\,\nu\int_\Omega\frac{\bt^2}{\o}\,\left|\d_x^3u\right|^2\,dx\,=\,\frac{1}{2}\int_\Omega\d_xu\,\big(\dd u\big)^2\,dx\,+\,\int_\Omega f\,\dd u\,dx\,.
\end{equation}
The core of the argument consists in finding suitable bounds for the terms on the right-hand side. Notice that, owing to the fact that we do not dispose of a strictly positive lower bound for $k$, the contribution of the viscosity term may appear of no importance. Yet, the piece of information coming from that term turn out to be essential for closing the estimates. 

Thus, let us bound the terms appearing on the right-hand side of the previous relation. We start by observing that
\begin{equation} \label{est:transp_u}
\left|\int_\Omega\d_xu\,(\dd u)^2\,dx\right|\,\lesssim\,\left\|\d_xu\right\|_{L^\infty}\,\left\|\dd u\right\|_{L^2}^2\,,
\end{equation}
which of course will be absobed by a Gr\"onwall-type argument.

It goes without saying that the control of the integral involving $f_1$ works absolutely in the same way. Thus, let us turn our attention to the integral involving $f_2$: after an integration by parts, this term can be coupled with the one coming from $f_3$, thus getting
\begin{align*}
\int_{\Omega}\left(f_2\,+\,f_3\right)\,\dd u\,dx\,&=\,-\,4\,\nu\int_\Omega\d_x\left(\frac{\bt^2}{\o}\right)\,\dd u\,\d_x^3u\,dx\,. 
\end{align*}
By computing
\[
 \d_x\left(\frac{\bt^2}{\o}\right)\,=\,2\,\frac{\bt}{\o}\,\d_x\bt\,-\,\frac{\bt^2}{\o^2}\,\d_x\o\,,
\]
in the end we have to bound two integrals. We start by writing
\[
\int_\Omega\frac{\bt}{\o}\,\d_x\bt\,\dd u\,\d_x^3u\,dx\,=\,\int_\Omega\frac{1}{\sqrt\o}\,\d_x\bt\,\dd u\,\frac{\bt}{\sqrt{\o}}\,\d_x^3u\,dx\,,
\]
and an application of Cauchy-Schwarz and Young inequalities gives us, for any $\de\in\,]0,1[\,$ to be chosen later, the estimate
\begin{align*}
\nu\left|\int_\Omega\frac{\bt}{\o}\,\d_x\bt\,\dd u\,\d_x^3u\,dx\right|\,\leq\,\de\,\nu\int_\Omega\frac{\bt^2}{\o}\,\left|\d_x^3u\right|^2\,dx\,+\,C(\de,\nu)\int_\Omega\frac{1}{\o}\,\left|\d_x\bt\right|^2\,\left|\dd u\right|^2\,dx\,.
\end{align*}
Similarly, we can control
\begin{align*}
\nu\left|\int_\Omega\frac{\bt^2}{\o^2}\,\d_x\o\,\dd u\,\d_x^3u\,dx\right|\,\leq\,\de\,\nu\int_\Omega\frac{\bt^2}{\o}\,\left|\d_x^3u\right|^2\,dx\,+\,C(\de,\nu)\int_\Omega\frac{\bt^2}{\o^3}\,\left|\d_x\o\right|^2\,\left|\dd u\right|^2\,dx\,.
\end{align*}
Therefore, using also \eqref{est:omega_L^inf}, in the end we get
\begin{align}
\left|\int_{\Omega}\left(f_2\,+\,f_3\right)\,\dd u\,dx\right|\,&\leq\,2\de\,\nu\int_\Omega\frac{\bt^2}{\o}\,\left|\d_x^3u\right|^2\,dx  \label{est:f_2-3} \\
&\qquad\,+\,C\,(1+t)^3\left(\left\|\d_x\bt\right\|^2_{L^\infty}\,+\,\|\bt\|^2_{L^\infty}\,\left\|\d_x\o\right\|_{L^\infty}^2\right)\,\left\|\dd u\right\|^2_{L^2}\,, \nonumber
\end{align}
where the value of $\de\in\,]0,1[\,$ will be fixed later on, and the constant $C>0$ depends only on
the values of $\de$ and of the various coefficients $(\nu,\al_1\ldots\al_4)$.
Of course, for $\de$ small enough, the first term on the right can be absorbed into the left-hand side of \eqref{est:dd-u_part}, whereas the second term can be controlled by the Gr\"onwall lemma.

It remains us to control the term depending on $f_4$. For this, we integrate by parts once, to get
\[
\int_\Omega f_4\,\dd u\,dx\,=\,-\nu\int_{\Omega}\dd\left(\frac{\bt^2}{\o}\right)\,\left|\dd u\right|^2\,dx\,-\,\nu\int_\Omega\dd\left(\frac{\bt^2}{\o}\right)\,\d_xu\,\d_x^3u\,dx\,.
\]
Observe that the first term on the right is analogous to the one coming from $f_2$, hence it can be controlled as we have done before (by computing a further integration by parts).
On the other hand, for bounding the second term in the right-hand side of the previous equation, we compute explicitly the second order derivative of the function $\bt^2/\o$: we get
\begin{equation} \label{eq:dd_k-o}
\dd\left(\frac{\bt^2}{\o}\right)\,=\,\frac{2}{\o}\,\big(\d_x\bt\big)^2\,+\,2\,\frac{\bt}{\o}\,\dd\bt\,-\,4\,\frac{\bt}{\o^2}\,\d_x\bt\,\d_x\o\,-\,\frac{\bt^2}{\o^2}\,\dd\o\,+\,2\,\frac{\bt^2}{\o^3}\,\left(\d_x\o\right)^2\,.
\end{equation}
So, we are reconducted to estimating five different integrals. Since the third-order term $\d_x^3u$ appears, the idea is to make a coefficient $\bt/\sqrt{\o}$ appear in front of it, as we have done for \eqref{est:f_2-3} above. To begin with, we estimate
\begin{align*}
\nu\,\left|\int_\Omega\frac{\bt}{\o}\,\dd \bt\,\d_xu\,\d_x^3u\,dx\right|\,&\leq\,\de\,\nu\int_\Omega\frac{\bt^2}{\o}\,\left|\d_x^3u\right|^2\,dx\,+\,C(\de,\nu)\int_\Omega\frac{1}{\o}\,\left|\d_xu\right|^2\,\left|\dd \bt\right|^2\,dx \\
\nu\,\left|\int_\Omega\frac{\bt^2}{\o^2}\,\dd \o\,\d_xu\,\d_x^3u\,dx\right|\,&\leq\,\de\,\nu\int_\Omega\frac{\bt^2}{\o}\,\left|\d_x^3u\right|^2\,dx\,+\,C(\de,\nu)\int_\Omega\frac{\bt^2}{\o^3}\,\left|\d_xu\right|^2\,\left|\dd \o\right|^2\,dx\,, 
\end{align*}
where $\de>0$ is as above. 
The same method would apply also to the third item in \eqref{eq:dd_k-o}, but this would give a bound which is not optimal. For this reason, we prefer to integrate
by parts once and write
\begin{align} \label{eq:second-term}
&\int_\Omega\frac{\bt}{\o^2}\,\d_x\bt\,\d_x\o\,\d_xu\,\d_x^3u\,dx \\
&\qquad\qquad\qquad =\,-\int_\Omega\frac{1}{\o^2}\,\big(\d_x\bt\big)^2\,\dd u\,\d_xu\,\d_x\o\,dx\,+\,2
\int_\Omega\frac{\bt}{\o^3}\,\d_x\bt\,\big(\d_x \o\big)^2\,\dd u\,\d_xu\,dx \nonumber \\
&\qquad\qquad\qquad\qquad\qquad -\,\int_\Omega\dfrac{\bt}{\o^2}\,\left(\dd \bt\,\d_x\o\,\d_xu\,+\,\d_x\bt\,\d_x\o\,\dd u\,+\,\d_x\bt\,\dd\o\,\d_xu\right)\,\dd u\,dx\,. \nonumber
\end{align}
All these terms can be controlled in a simple way. For instance, using also \eqref{est:omega_L^inf} and the interpolation inequality
\eqref{est:interpolation}, for the first term we have
\begin{align*}
\left|\int_\Omega\frac{1}{\o^2}\,\big(\d_x\bt\big)^2\,\dd u\,\d_xu\,\d_x\o\,dx\right|\,&\lesssim\,\frac{1}{\big(\o_{\min}(t)\big)^2}\,\left\|\dd u\right\|_{L^2}\,\left\|\d_xu\right\|_{L^\infty}\,\left\|\d_x\o\right\|_{L^\infty}\,\left\|\d_x\bt\right\|^2_{L^4} \\
&\lesssim\,(1+t)^2\,\left\|\left(\dd u\,,\,\dd\bt\right)\right\|^2_{L^2}\,\left\|\d_xu\right\|_{L^\infty}\,\|\bt\|_{L^\infty}\,
\left\|\d_x\o\right\|_{L^\infty}\,.
\end{align*}
Similarly, we can bound the second term as follows:
\begin{align*}
\left|\int_\Omega\frac{\bt}{\o^3}\,\d_x\bt\,\big(\d_x \o\big)^2\,\dd u\,\d_xu\,dx\right|\,&\lesssim\,\frac{1}{\big(\o_{\min}(t)\big)^3}\,\|\bt\|_{L^\infty}\,\left\|\dd u\right\|_{L^2}\,\left\|\d_xu\right\|_{L^\infty}\,\left\|\d_x\bt\right\|_{L^\infty}\,\left\|\d_x\o\right\|^2_{L^4} \\
&\lesssim\,\frac{\o_{\max}(t)}{\left(\o_{\min}(t)\right)^3}\,\|\bt\|_{L^\infty}\,\left\|\left(\dd u\,,\,\dd\o\right)\right\|^2_{L^2}\,\left\|\d_xu\right\|_{L^\infty}\,\left\|\d_x\bt\right\|_{L^\infty} \\
&\lesssim\,(1+t)^2\,\|\bt\|_{L^\infty}\,\left\|\left(\dd u\,,\,\dd\o\right)\right\|^2_{L^2}\,\left\|\d_xu\right\|_{L^\infty}\,\left\|\d_x\bt\right\|_{L^\infty}\,.
\end{align*}
The control of the terms in the third line of \eqref{eq:second-term} is even more straightforward. All in all, we finally get
\begin{align*}
\left|\int_\Omega\frac{\bt}{\o^2}\,\d_x\bt\,\d_x\o\,\d_xu\,\d_x^3u\,dx\right|\,\lesssim\,(1+t)^2\,\|\bt\|_{L^\infty}\,
\left\|\left(\d_xu\,,\,\d_x\o\,,\,\d_x\bt\right)\right\|^2_{L^\infty}\,\left\|\left(\dd u\,,\,\dd\o\,,\,\dd\bt\right)\right\|^2_{L^2}\,.
\end{align*}
Finally, the terms coming from the first and fifth items of \eqref{eq:dd_k-o} can be treated as the previous one, after an integration by parts: we obtain
\begin{align*}
\left|\int_\Omega\frac{1}{\o}\,\big(\d_x\bt\big)^2\,\d_xu\,\d_x^3u dx\right|\,&\lesssim\,(1+t)^2\,\left(1\,+\,\|\bt\|_{L^\infty}\right)\,
\left\|\left(\d_xu\,,\,\d_x\o\,,\,\d_x\bt\right)\right\|^2_{L^\infty}\,\left\|\left(\dd u\,,\,\dd\bt\right)\right\|^2_{L^2} \\
\left|\int_\Omega\frac{\bt^2}{\o^3}\,\big(\d_x\o\big)^2\,\d_xu\,\d_x^3u dx\right|\,&\lesssim\,(1+t)^3\,\left(1\,+\,\|\bt\|^2_{L^\infty}\right)\,
\left\|\left(\d_xu\,,\,\d_x\o\,,\,\d_x\bt\right)\right\|^2_{L^\infty}\,\left\|\left(\dd u\,,\,\dd\o\right)\right\|^2_{L^2} \,.
\end{align*}
Putting all those inequalities together, we finally deduce the sought control for the $f_4$ term:
\begin{align}
\left|\int_{\Omega}f_4\,\dd u\,dx\right|\,&\leq\,C_1\,\de\,\nu\int_\Omega\frac{\bt^2}{\o}\,\left|\d_x^3u\right|^2\,dx  \label{est:f_4} \\
&\quad\,+\,C_2\,(1+t)^3\left(1+\|\bt\|^2_{L^\infty}\right)\,\left\|\left(\d_xu,\d_x\o,\d_x\bt\right)\right\|^2_{L^\infty}\,\left\|\left(\dd u,\dd\o,\dd\bt\right)\right\|^2_{L^2}\,, \nonumber
\end{align}
where the value of $\de\in\,]0,1[\,$ will be fixed later on, the constant $C_1>0$ is ``universal'' (independent of all the data and parameters appearing in the equations) and the constant $C_2>0$ depends only on
the values of $\de$ and of the various coefficients $(\nu,\al_1\ldots\al_4)$.

In the end, inserting the bounds we have obtained in \eqref{est:transp_u}, \eqref{est:f_2-3} and \eqref{est:f_4} into \eqref{est:dd-u_part} and choosing $\de>0$ small enough, we find a bound for $\dd u$:
\begin{align} 
 &\frac{1}{2}\,\frac{d}{dt}\left\|\dd u\right\|_{L^2}^2\,+\,\nu\int_\Omega\frac{\bt^2}{\o}\,\left|\d_x^3u\right|^2\,dx \label{est:dd-u} \\
&\qquad\qquad\qquad
\lesssim\,(1+t)^3\left(1+\|\bt\|^2_{L^\infty}\right)\,\left(1\,+\,\left\|\left(\d_xu,\d_x\o,\d_x\bt\right)\right\|^2_{L^\infty}\right)\,\left\|\left(\dd u,\dd\o,\dd \bt\right)\right\|^2_{L^2}
 \,. \nonumber
\end{align}

\subsubsection{The $H^2$ estimate for $\o$} \label{sss:omega}

We now perform $H^2$ estimates on $\o$. For this, it is enough to perform $L^2$ estimates on equation \eqref{eq:dd-omega} for $\dd\o$. In this way, we easily get
\begin{align*} 
 \frac{1}{2}\,\frac{d}{dt}\left\|\dd \o\right\|_{L^2}^2+\al_1\int_\Omega\frac{\bt^2}{\o}\,\left|\d_x^3\o\right|^2\,dx+2\al_2\int_\Omega\o\,\left|\dd\o\right|^2\,dx\,=\,\frac{1}{2}\int_\Omega\d_xu\,\big(\dd \o\big)^2\,dx\,+\,\int_\Omega g\,\dd\o\,dx\,.
\end{align*}

In the estimate of the right-hand side of the equation above, only few terms change with respect to the ones we have already treated in the previous paragraph. First of all, we put together the transport term, \tsl{i.e.} the first term appearing on the right, with $g_1$: we find
\[
\left|\int_\Omega\left(\frac12\,\d_xu\,-\,2\,\d_xu\right)\,\left|\dd\o\right|^2\,dx\right|\,\lesssim\,\left\|\d_xu\right\|_{L^\infty}\,\left\|\dd\o\right\|_{L^2}^2\,.
\]

Next, we easily see that the control of the terms $g_3$, $g_4$ qand $g_5$ is analogous to the one of $f_2$, $f_3$ and $f_4$ respctively, up to change $u$ by $\o$ and $\nu$ by $\al_1$ everywhere in those computations. We infer the bound
\begin{align*}
\left|\int_\Omega\big(g_3+g_4+g_5\big)\,\dd\o\,dx\right|\,&\leq\,C_1\,\de\,\alpha_1\int_\Omega\frac{\bt^2}{\o}\,\left|\d_x^3\o\right|^2\,dx \\
&\;+\,C_2\,(1+t)^3\left(1+\|\bt\|^2_{L^\infty}\right)\,\left\|\left(\d_x\o,\d_x\bt\right)\right\|^2_{L^\infty}\,\left\|\left(\dd\o,\dd\bt\right)\right\|^2_{L^2}\,,
\end{align*}
where $\de\in\,]0,1[$ can be taken arbitrarily small and the constants $C_1>0$ and $C_2>0$ are as in \eqref{est:f_4}.

Thus, it remains us to bound the terms coming from $g_2$ and $g_6$; however, both controls are fairly straightforward. As a matter of fact, we have
\begin{align*}
\left|\int_\Omega g_2\,\dd\o\,dx\right|\,&\lesssim\,\left\|\d_x\o\right\|_{L^\infty}\,\left\|\left(\dd u,\dd\o\right)\right\|_{L^2}^2 \\
\left|\int_\Omega g_6\,\dd\o\,dx\right|\,&\lesssim\,\left\|\d_x\o\right\|_{L^4}^2\,\left\|\dd\o\right\|_{L^2}\;\lesssim\;\o_{\max}(t)\,\left\|\dd\o\right\|_{L^2}^2\,,
\end{align*}
where, in the last line, we have also used the interpolation inequality \eqref{est:interpolation}.

In the end, we have obtained the following bound, analogous to \eqref{est:dd-u}:
\begin{align} 
 &\frac{1}{2}\,\frac{d}{dt}\left\|\dd\o\right\|_{L^2}^2\,+\,\al_1\int_\Omega\frac{\bt^2}{\o}\,\left|\d_x^3\o\right|^2\,dx\,+\,2\,\al_2\int_\Omega\o\,\left|\dd\o\right|^2\,dx \label{est:dd-o} \\
&\qquad\qquad\lesssim\,(1+t)^3\left(1+\|\bt\|^2_{L^\infty}\right)\,\left(1\,+\,\left\|\left(\d_xu,\d_x\o,\d_x\bt\right)\right\|^2_{L^\infty}\right)\,\left\|\left(\dd u,\dd\o,\dd\bt\right)\right\|^2_{L^2}
 \,. \nonumber
\end{align}

\subsubsection{The $H^2$ estimate for $\bt$} \label{sss:k}

Let us switch to the control of $\dd \beta$. Once again, we start with an energy estimate for equation \eqref{eq:dd-k}, thus finding
\begin{align*} 
 \frac{1}{2}\,\frac{d}{dt}\left\|\dd\bt\right\|_{L^2}^2+\al_3\int_\Omega\frac{\bt^2}{\o}\,\left|\d_x^3\bt\right|^2\,dx+\,\frac{1}{2}\int_\Omega\o\,\left|\dd\bt\right|^2\,dx\,=\,\frac{1}{2}\int_\Omega\d_xu\,\big(\dd\bt\big)^2\,dx\,+\,\int_\Omega h\,\dd\bt\,dx\,.
\end{align*}

As before, the transport term has to be joined with $h_1$: this gives us
\[
\left|\int_\Omega\left(\frac12\,\d_xu\,-\,2\,\d_xu\right)\,\left|\dd\bt\right|^2\,dx\right|\,\lesssim\,\left\|\d_xu\right\|_{L^\infty}\,\left\|\dd\bt\right\|_{L^2}^2\,.
\]

On the other hand, the group of terms $h_5$, $h_6$ and $h_7$ have to be handled in the same way as $f_2$, $f_3$ and $f_4$ above. \tsl{Mutatis mutandis}, we gather the control
\begin{align*}
\left|\int_\Omega\big(h_5+h_6+h_7\big)\,\dd\bt\,dx\right|\,&\leq\,C_1\,\de\,\alpha_3\int_\Omega\frac{\bt^2}{\o}\,\left|\d_x^3\bt\right|^2\,dx \\
&\qquad
+\,C_2\,(1+t)^3\left(1+\|\bt\|^2_{L^\infty}\right)\,\left\|\left(\d_x\o,\d_x\bt\right)\right\|^2_{L^\infty}\,\left\|\left(\dd\o,\dd\bt\right)\right\|^2_{L^2}\,.
\end{align*}
The previous bound holds true for any $\de\in\,]0,1[\,$; the precise value of $\de$ will be fixed later on. In addition, the positive constants $C_1$
and $C_2$ satisfy the same properties as in \eqref{est:f_4} above.

Next, arguing as for $g_6$ and using the interpolation inequality \eqref{est:interpolation}, we easily get
\begin{align*}
\left|\int_\Omega \left(h_2\,+\,h_4\right)\,\dd\bt\,dx\right|\,&\lesssim\,\left(\left\|\d_x\bt\right\|_{L^\infty}\,+\,\|\bt\|_{L^\infty}\right)\,\left\|\left(\dd u,\dd\o,\dd\bt\right)\right\|_{L^2}^2 \\
\left|\int_\Omega h_3\,\dd\bt\,dx\right|\,&\lesssim\,\left\|\d_x\o\right\|_{L^4}\,\left\|\d_x\bt\right\|_{L^4}\,\left\|\dd\bt\right\|_{L^2}\;\lesssim\;\left(\o_{\max}(t)\,+\,\|\bt\|_{L^\infty}\right)\,\left\|\left(\dd\o,\dd\bt\right)\right\|_{L^2}^2\,.
\end{align*}

Hence, it remains us to bound the terms appearing in the last two lines of the definition of $h$. We start from $h_{11}$: using the usual trick based on Cauchy-Schwarz and Young inequalities, for $\de\in\,]0,1[\,$ as above, we bound
\begin{align*}
\left|\int_\Omega h_{11}\,\dd\bt\,dx\right|\,&\leq\,\de\,\nu\int_\Omega\frac{\bt^2}{\o}\,\left|\d_x^3u\right|^2\,dx\,+\,C(\de,\nu,\al_4)\int_\Omega \frac{1}{\o}\,\left|\d_xu\right|^2\,\left|\dd\bt\right|^2\,dx \\
&\leq\,\de\,\nu\int_\Omega\frac{\bt^2}{\o}\,\left|\d_x^3u\right|^2\,dx\,+\,C\,(1+t)\,\left\|\d_xu\right\|^2_{L^\infty}\,\left\|\dd\bt\right\|^2_{L^2}\,.
\end{align*}
Next, we consider $h_{10}$: for it, we can use an integration by parts to write
\[
\int_\Omega h_{10}\,\dd\bt\,dx\,=\,-\,2\,\al_4\int_\Omega\frac{\bt}{\o}\,\dd u\,\d_x^3u\,\d_x\bt\,dx\,-\,
\al_4\int_\Omega\d_x\left(\frac{\bt}{\o}\right)\,\left(\dd u\right)^2\,\d_x\bt\,dx\,.
\]
It is easily seen that the first term on the right-hand side can be handled in the same way as for $h_{11}$, thus producing
\[
\alpha_4\,\left|\int_\Omega\frac{\bt}{\o}\,\dd u\,\d_x^3u\,\d_x\bt\,dx\right|\,\leq\,\de\,\nu\int_\Omega\frac{\bt^2}{\o}\,\left|\d_x^3u\right|^2\,dx\,+\,C(\de,\nu,\alpha_4)\,(1+t)\,\left\|\d_x\bt\right\|^2_{L^\infty}\,\left\|\dd u\right\|^2_{L^2}\,,
\]
whereas for the second term one can explicitly compute the space derivative of $\bt/\o$ and obtain
\begin{align*}
\left|\int_\Omega\d_x\left(\frac{\bt}{\o}\right)\,\left(\dd u\right)^2\,\d_x\bt\,dx\right|\,\lesssim\,(1+t)^2\,\left(1+\|\bt\|_{L^\infty}\right)\,\left\|\left(\d_x\o,\d_x\bt\right)\right\|^2_{L^\infty}\,\left\|\dd u\right\|_{L^2}^2\,.
\end{align*}
The control of $h_9$ works in a pretty similar way: after computing $\d_x(\bt/\o)$, one easily gets
\[
\left|\int_\Omega h_9\,\dd\bt\,dx\right|\,\lesssim\,(1+t)^2\,\left(1\,+\,\|\bt\|_{L^\infty}\right)\,\left\|\left(\d_xu,\d_x\o,\d_x\bt\right)\right\|^2_{L^\infty}\,\left\|\left(\dd u,\dd \bt\right)\right\|_{L^2}^2\,.
\]
Next, in order to estimate the term presenting $h_8$, we compute 
\[
\d_x^2\left(\frac{\bt}{\o}\right)\,=\,\frac{1}{\o}\,\dd\bt\,-\,\frac{2}{\o^2}\,\d_x\bt\,\d_x\o\,-\,\frac{\bt}{\o^2}\,\dd\o\,+\,2\,\frac{\bt}{\o^3}\,\big(\d_x\o\big)^2
\]
and reduce our problem to bound four different integrals. For this, we are going to argue in a very similar way as done for estimating the terms appearing in \eqref{eq:second-term}. On the one hand, we have
\begin{align*}
\left|\int_\Omega\frac{1}{\o}\,\left|\d_xu\right|^2\,\left(\dd\bt\right)^2 \,dx\right|\,&\lesssim\,
(1+t)\,\left\|\d_xu\right\|^2_{L^\infty}\,\left\|\dd\bt\right\|^2_{L^2} \\
\left|\int_\Omega\frac{\bt}{\o^2}\,\dd\o\,\left|\d_xu\right|^2\,\dd\bt\,dx\right|\,&\lesssim\,
(1+t)^2\,\|\bt\|_{L^\infty}\,\left\|\d_xu\right\|^2_{L^\infty}\,\left\|\left(\dd\o,\dd\bt\right)\right\|^2_{L^2}\,.
\end{align*}
On the other hand, for the terms presenting a mixing of first order derivatives, we use the interpolation inequality \eqref{est:interpolation} to deduce
\begin{align*}
&\left|\int_\Omega\left(\frac{1}{\o^2}\,\d_x\bt\,\d_x\o+\frac{\bt}{\o^3}\left(\d_x\o\right)^2\right)\left|\d_xu\right|^2\dd\bt\,dx\right|\,\lesssim\,(1+t)^2\left(1+\|\bt\|_{L^\infty}\right)\left\|\d_xu\right\|_{L^\infty}^2\left\|\left(\dd\o,\dd\bt\right)\right\|^2_{L^2}.
\end{align*}

Finally, we see that the terms arising from the last line of the definition of $h$ can be handled as $h_8\ldots h_{11}$, up to replace $u$ by $\bt$
and $\alpha_4$ by $\alpha_3$ everywhere in the computations.

In the end, putting all those inequalities together and choosing $\de>0$ small enough, we arrive at the following estimate for $\bt$:
\begin{align} 
 &\frac{1}{2}\,\frac{d}{dt}\left\|\dd\bt\right\|_{L^2}^2\,+\,\al_3\int_\Omega\frac{\bt^2}{\o}\,\left|\d_x^3\bt\right|^2\,dx\,+\,\int_\Omega\o\,\left|\dd\bt\right|^2\,dx \label{est:dd-k} \\
&\qquad\qquad\qquad\lesssim\,(1+t)^3\left(1+\|\bt\|^2_{L^\infty}\right)\,\left(1+\left\|\left(\d_xu,\d_x\o,\d_x\bt\right)\right\|^2_{L^\infty}\right)\,\left\|\left(\dd u,\dd\o,\dd\bt\right)\right\|^2_{L^2}\,, \nonumber
\end{align}
for an implicit multiplicative constant $C>0$, depending only on the chosen value of $\de>0$ and on
the various parameters $\big(\nu,\al_1\ldots \al_4\big)$.

\subsection{Closing the estimates} \label{ss:closing}

In order to close the estimates in some small time interval $[0,T]$, we introduce some convenient notations. For all $t\geq0$, we set
\[
E(t)\,:=\,\left\|u(t)\right\|_{H^2}^2\,+\,\left\|\o(t)\right\|_{H^2}^2\,+\,\left\|\bt(t)\right\|_{H^2}^2
\]
to be the energy of the solution. Our goal is to bound it in terms of $E(0)$, which is defined as the same quantity, computed on the initial datum $\big(u_0,\o_0,\bt_0\big)$. We also define
\begin{align*}
 E_0(t)\,&:=\,\left\|u(t)\right\|_{L^2}^2\,+\,\left\|\o(t)\right\|_{L^2}^2\,+\,\left\|\bt(t)\right\|_{L^2}^2 \\ 
 E_2(t)\,&:=\,\left\|\dd u(t)\right\|_{L^2}^2\,+\,\left\|\dd\o(t)\right\|_{L^2}^2\,+\,\left\|\dd \bt(t)\right\|_{L^2}^2\,,
\end{align*}
and we denote by $E_0(0)$ and $E_2(0)$ the same quantities, when computed on the initial data.
Therefore, in view of what we have said at the beginning of Subsection \ref{ss:higher}, we have
\begin{equation} \label{eq:en_equiv}
\forall\,t\geq0\,,\qquad\qquad E(t)\,\sim\,E_0(t)\,+\,E_2(t)\,.
\end{equation}

First of all, we observe that, from \eqref{est:u}, \eqref{est:omega} and \eqref{est:beta_L^2}, we deduce that
\begin{equation} \label{est:E_0}
\forall\,t\,\geq\,0\,,\qquad\qquad E_0(t)\,\leq\,C_1\,E_0(0)\;\exp\left(C_1\int^t_0\left\|\d_xu(\t)\right\|_{L^\infty}\,d\t\right)\,,
\end{equation}
for a ``universal'' numerical constant $C_1>0$ independent of any parameters and data of the problem.

Next, we put estimates \eqref{est:dd-u}, \eqref{est:dd-o} and \eqref{est:dd-k} together: in this way, after setting
\begin{equation} \label{def:A}
A(t)\,:=\,(1+t)^3\left(1+\|\bt\|^2_{L^\infty}\right)\,\left(1\,+\,\left\|\left(\d_xu,\d_x\o,\d_x\bt\right)\right\|^2_{L^\infty}\right)\,,
\end{equation}
we get, for all $t>0$, the inequality
\begin{align*}
&\frac{d}{dt}E_2(t)\,+\int_\Omega\frac{\bt^2}{\o}\left(\left|\d_x^3u\right|^2+\left|\d_x^3\o\right|^2+\left|\d_x^3\bt\right|^2\right)dx\,+\int_\Omega\o\left(\left|\dd\o\right|^2+\left|\dd\bt\right|^2\right)dx\,\leq\,C_2\,A(t)\,E_2(t)\,,
\end{align*}
for a new constant $C_2\,=\,C_2(\nu,\al_1\ldots\al_4)\,>\,0$ only depending on the quantity in the
brackets.
An application of the Gr\"onwall lemma thus yields 
\begin{equation} \label{est:E_2}
\forall\,t\geq0\,,\qquad\qquad E_2(t)\,\leq\,E_2(0)\,\exp\left(C_2\int^t_0A(\t)\,d\t\right)\,.
\end{equation}

In view of the equivalence stated in \eqref{eq:en_equiv}, summing up inequalities \eqref{est:E_0} and \eqref{est:E_2}  thus yields
\begin{equation} \label{est:E_int}
\forall\,t\geq0\,,\qquad\qquad E(t)\,\leq\,E(0)\,\exp\left(\left(C_1+C_2\right)\int^t_0A(\t)\,d\t\right)\,.
\end{equation}

At this point, if we define the time $T>0$ such that
\[
 T\,:=\,\sup\left\{t>0\;\bigg|\; \int^t_0A(\t)\,d\t\,\leq\,2\,\log2\right\}\,,
\]
we deduce from \eqref{est:E_int} that, for a new suitable constant $C_3>0$, only depending on the parameters $(\nu,\al_1\ldots\al_4)$, we have
\begin{equation} \label{est:E}
 \sup_{t\in[0,T]}E(t)\,\leq\,C_3\,E(0)\,.
\end{equation}

As a direct consequence of the previous computations, we also get the blow-up criterion. Indeed, from \eqref{est:E_int} we see that, if $T^*<+\infty$ and if we have
\[
\int^{T^*}_0A(\t)\,d\t\,<\,+\infty\,,
\]
then $E(T^*)<+\infty$. 
Thus the solution cannot blow up at time $T^*>0$, and hence it can be continued, by a standard argument, into a solution of system \eqref{eq:kolm} with the same regularity
(notice that, for this step, we also need the uniqueness of the solutions in the considered functional framework). At this point, we can remove the factor
$(1+t)^3$ in the continuation criterion, by an easy argument. In addition, we observe that, thanks to \eqref{est:interp_2}, where we take $p=2$, we can control
\[
\int^{T^*}_0\left\|\bt\right\|^2_{L^\infty}\,dt\,\lesssim\,\int^{T^*}_0\left(\oline\bt^2\,+\,\left\|\bt\right\|_{L^2}\,\left\|\d_x\bt\right\|_{L^2}\right)\,dt
\,\lesssim\,\left\|\bt_0\right\|_{L^2}^2\,T^*\,+\,\left\|\bt_0\right\|_{L^2}\int^{T^*}_0\left\|\d_x\bt\right\|_{L^\infty}\,dt\,,
\]
where we have used also \eqref{est:beta_L^2}. Therefore, we see that the previous integral is finite, if the integral appearing in \eqref{cont-cont}
is finite.
The continuation criterion is then proved.

Finally, let us derive the claimed lower bound on the lifespan of the solutions. For this, we recall the definition of the function $A(t)$, given in \eqref{def:A}.
Arguing as before and making use of inequality \eqref{est:interp_2} with $p=2$ for $\d_xu$, $\d_x\o$ and $\d_x\bt$, we find that
\begin{align*}
\left\|\bt\right\|^2_{L^\infty}\,&\lesssim\,\oline\bt^2\,+\,\left\|\bt\right\|_{L^2}\,\left\|\d_x\bt\right\|_{L^2}\,\lesssim\,E(t) \\
\left\|\d_xu\right\|^2_{L^\infty}\,&\lesssim\,\left\|\d_xu\right\|_{L^2}\,\left\|\dd u\right\|_{L^2}\,\lesssim\,E(t)\,,
\end{align*}
and analogous inequalities for $\d_x\o$ and $\d_x\bt$.
From those estimates, we duduce that, for any $t\in[0,T]$, one has
\begin{align*}
\int^t_0A(\t)\,d\t\,&\leq\,(1+t)^3\,\int^t_0\left(1+\|\bt\|^2_{L^\infty}\right)\,
\left(1\,+\,\left\|\left(\d_xu,\d_x\o,\d_xk\right)\right\|^2_{L^\infty}\right)\,d\t \\
&\lesssim\,(1+t)^3\,\int^t_0\big(1\,+\,E(\t)\big)^2\,d\t\,. 
\end{align*}
Owing to \eqref{est:E}, in the end we get the following bound:
\[
\forall\,t\in[0,T]\,,\qquad\qquad\int^t_0A(\t)\,d\t\,\lesssim\,(1+t)^3\,t\,\Big(\max\big\{1,E(0)\big\}\Big)^2\,.
\]
Therefore, using the definition of $T$, by continuity we discover that
\[
C\,(1+T)^3\,\Big(\max\big\{1,E(0)\big\}\Big)^2\,T\,\geq\,\int^T_0A(\t)\,d\t\,=\,2\,\log2\,,
\]
for a suitable positive constant $C$. At this point, if $T\geq1$, we are done. Assume, instead, that $T<1$: then, the previous inequality
implies that
\[
\wtilde C\,\Big(\max\big\{1,E(0)\big\}\Big)^2\,T\,\geq\,\int^T_0A(\t)\,d\t\,=\,2\,\log2\,,
\]
for a new constant $\wtilde C\,=\,8C$, which implies the claimed lower bound on the time $T$.


\section{Proof of the main results} \label{s:proof}

In this section, we prove our main results, namely Theorems \ref{t:local} and \ref{t:blow-up}.
First we deal with the local in time well-posedness issue for system \eqref{eq:kolm}. Then, we show that, under suitable symmetry assumptions, the solutions we have previously constructed in fact blow up in finite time.

\subsection{Proof of the local existence} \label{ss:local}

Here we prove the local in time well-posedness of system \eqref{eq:kolm}, contained in Theorem \ref{t:local}. Once again, we recall that we only show existence and uniqueness of solutions at $H^2$ level of regularity; the case of higher $H^m$ regularity follows by similar arguments.

For notational convenience, throughout all this subsection we adopt the following convention:
if $\big(f_\veps\big)_{\varepsilon > 0}$ is a sequence of functions in a normed space $X$, such that $\big(f_\veps\big)_{\veps>0}$ is \emph{bounded} in $X$,  we simply write $\big(f_\veps\big)_{\veps > 0} \subset X$.


\paragraph{Construction of approximated solutions.} To construct solutions to system \eqref{eq:kolm} under the assumptions of Theorem \ref{t:local}, the first step is to remove the singularity created by the possible vanishing of the mean turbulent kinetic energy $k$. At the same
time, we will also regularise the initial datum, in order to get smooth approximate solutions, for which the computations performed
in Section \ref{s:a-priori} are fully justified.

For all $t>0$, let $\mathcal{H}_t$ be the periodic heat kernel at time $t$.
Then, for any $\veps\in\,]0,1]$, we define
\[
k_{0,\veps}\,:=\,\mathcal{H}_\veps *\left(\sqrt{k_0}\,+\,\veps\right)^2\,=\,\mc H_\veps *k_0\,+\,2\,\veps\,\mc H_\veps *\sqrt{k_0}\,+\,\veps^2\,,
\]
where the convolution is taken only with respect to the space variables.
Observe that, since $\sqrt{k_0}\in H^2(\Omega)$ and $H^2(\Omega)$ is an algebra, also $k_0\in H^2(\Omega)$.
Therefore, we have 
\begin{equation} \label{u-est:k_0}
\big(k_{0,\veps}\big)_\veps\,\subset\,H^2(\Omega)\,,\qquad\qquad\mbox{ with }\qquad k_{0,\veps}\,\geq\,\veps^2\,>\,0\,.
\end{equation}

Next, for any fixed $\veps\in\,]0,1]$, we solve system \eqref{eq:kolm} with respect to the smooth initial datum
$\big(\mathcal{H}_\veps *u_{0}\,,\,\mathcal{H}_\veps *\o_{0}\,,\,k_{0,\veps}\big)$. Notice that
\[
\big(\mathcal{H}_\veps *u_{0}\,,\,\mathcal{H}_\veps *\o_{0}\,,\,k_{0,\veps}\big)_{\veps\in\,]0,1]}\,\subset\,
 H^2(\Omega)\,\times\,H^2(\Omega)\,\times\,H^2(\Omega)\,.
\]
By employing, for instance, the method of \cite{Kos-Kub} to the one-dimensional setting, for any $\veps>0$ we construct a
unique smooth local in time solutions $\big(u_{\veps},\o_{\veps},k_{\veps}\big)$, defined in some time interval
$[0,T_{\veps}]$. We refer to Theorem 1 in \cite{Kos-Kub} for a precise statement. Notice that the arguments there, based on Galerkin approximation,
are easily adaptable to prove propagation of higher order regularities.


Thus, by repeating the computations of Section \ref{s:a-priori},
we can establish some \emph{uniform} properties for the family of smooth solutions $\big(u_{\veps},\o_{\veps},k_{\veps}\big)_{\veps\in\,]0,1]}$.

First of all, owing to the lower bound \eqref{est:lower-T}, we observe that
all the solutions are defined in a common time interval $[0,T]$, where
\[
T\,:=\,\inf_{\veps\in\,]0,1]}T_\veps\,>\,0\,.
\]

Next, from \eqref{est:omega_L^inf}, we gather that
\begin{equation} \label{est:o_e-inf}
\big(\o_{\veps}\big)_\veps\,\subset\,L^\infty_T(L^\infty)\,,\qquad\qquad\mbox{ with }\qquad
0\,<\,\frac{\o_*}{\o_*\,\alpha_2\,t\,+\,1}\,\leq\,\omega_{\veps}(t,x)\,\leq\,\frac{\o^*}{\o^*\,\alpha_2\,t\,+\,1}\,,
\end{equation}
whereas \eqref{est:k-inf} implies that $k_{\veps}(t,x)\,\geq\,0$ for all $(t,x)\in[0,T]\times\Omega$.

Moreover, from \eqref{est:E}, we infer that
\begin{equation} \label{est:sol_e}
\left(u_{\veps},\o_{\veps},\sqrt{k_{\veps}}\right)_\veps\,\subset\,L^\infty_T(H^{2})\,\times\,L^\infty_T(H^{2})\,\times\,L^\infty_T(H^2)
\end{equation}
and also that
\begin{equation} \label{est:d_x^3_e}
\left(\sqrt{\frac{k_{\veps}}{\o_{\veps}}}\,\d_x^{3}u_{\veps}\right)_\veps,\quad
\left(\sqrt{\frac{k_{\veps}}{\o_{\veps}}}\,\d_x^{3}\o_{\veps}\right)_\veps,\quad
\left(\sqrt{\frac{k_{\veps}}{\o_{\veps}}}\,\d_x^{3}\sqrt{k_{\veps}}\right)_\veps\quad\subset\quad L^2_T(L^2)\,.
\end{equation}
As a consequence of \eqref{est:sol_e}, we also get that
\[
\left(k_{\veps}\right)_{\veps}\,\subset\,L^\infty_T(H^{2})\,.
\]


\paragraph{Stability analysis and convergence.}
Thanks to the previous uniform embedding properties, we can extract a weakly converging subsequence. More precisely, owing to \eqref{est:sol_e},
we deduce the existence of functions $u$, $\o$ and $k$, all belonging to the space $L^\infty_T(H^2)$, such that,
up to the extraction of a subsequence (not relabelled here), one has the convergences
\[
u_\veps\,\stackrel{*}{\rightharpoonup}\,u\,,\quad \o_\veps\,\stackrel{*}{\rightharpoonup}\,\o\quad \mbox{ and }\quad
k_\veps\,\stackrel{*}{\rightharpoonup}\,k\qquad\qquad \mbox{ in }\qquad L^\infty_T(H^2)\,.
\]
The basic idea is that the triplet $\big(u,\o,k\big)$ identified above is the sought solution to the Kolmogorov system \eqref{eq:kolm}. Nonetheless, of course the previous weak convergence properties are not enough to assert this, owing to the presence of non-linear terms in the equations.
Therefore, we need some compactness properties for the sequence of solutions $\big(u_\veps,\o_\veps,k_\veps\big)_\veps$: this is the goal of the present paragraph.

Before going on, let us immediately clarify that it is enough to prove that the triplet $\big(u,\o,k\big)$ solves \eqref{eq:kolm} in the weak sense.
Similarly, in order to verify this, it is enough to pass to the limit in the weak formulation of the equations satisfied by $\big(u_\veps,\o_\veps,k_\veps\big)$.
Here, the weak formulation we refer to is the classical one: for instance, concerning the approximate velocity fields $u_\veps$,
we have
\begin{align*}
\iint_{\R_+\times\Omega}\left(-\,u_\veps\,\d_t\phi\,-\,\frac{1}{2}\,u_\veps^2\,\d_x\phi\,+\,\nu\,\frac{k_\veps}{\o_\veps}\,\d_xu_\veps\,\d_x\phi\right)\,dx\,dt\,=\,
\int_\Omega u_{0,\veps}\,\phi(0,\cdot)\,dx
\end{align*}
for any test-function $\phi\in C^\infty_c(\R_+\times\Omega)$, and analogous relations holds for $\big(\o_\veps\big)_\veps$ and $\big(k_\veps\big)_\veps$.

We start by analysing the equation for $u_\veps$, which reads
\[
 \d_tu_\veps\,=\,-\,u_\veps\,\d_xu_\veps\,+\,\nu\,\d_x\left(\frac{k_\veps}{\o_\veps}\,\d_xu_\veps\right)\,.
\]
From the previous uniform bounds \eqref{est:o_e-inf} and \eqref{est:sol_e}, it is easy to deduce that $\big(\d_tu_\veps\big)_\veps$
is uniformly bounded in the space $L^\infty_T(H^{-1})$, hence
\begin{equation} \label{ub:u_e}
\big(u_\veps\big)_\veps\,\subset\,L^\infty_T(H^2)\,\cap\,W^{1,\infty}_T(H^{-1})\,.
\end{equation}
An application of the Ascoli-Arzel\`a theorem (see \tsl{e.g.} Lemma 3.6 in \cite{Tsai}) thus implies that
$\big(u_\veps\big)_\veps$ is compact in the space $C_T(H^{-1})$. Then, up to the extraction of a further subsequence, one gathers the strong convergence
\begin{equation} \label{conv:strong-u}
u_\veps\,\longrightarrow\,u\qquad\qquad\mbox{ in }\qquad C_T(H^s)\qquad \forall\;0\,\leq\,s\,<\,2\,,
\end{equation}
where we have also used an interpolation argument between the strong convergence in $C_T(H^{-1})$ and the uniform bounds in $L^\infty_T(H^2)$.

It goes without saying that a similar argument applies also to the equations satisfied for $\o_\veps$ and $k_\veps$, yielding strong convergence
of $\big(\o_\veps\big)_\veps$ and $\big(k_\veps\big)_\veps$ to, respectively, $\o$ and $k$ in $C_T(H^s)$, for any $s<2$.

In particular, from \eqref{conv:strong-u} and Sobolev embeddings, we deduce the pointwise convergences
\begin{equation} \label{conv:point-u}
u_\veps\,\longrightarrow\,u\qquad \mbox{ and }\qquad \d_xu_\veps\,\longrightarrow\,\d_xu\qquad\qquad \mbox{ almost everywhere in }\; [0,T]\times\Omega\,,
\end{equation}
and analogous pointwise convergence properties also for the sequences $\big(\o_\veps\big)_\veps$, $\big(k_\veps\big)_\veps$ and
their first order space derivatives.

At this point, we make use of \eqref{conv:point-u} for $u$, $\o$ and $k$, and of the strong convergence of the initial datum
$k_{0,\veps}\,\longrightarrow\,k_0$. Then,
it is easy to pass to the limit in the weak formulation of the equations for
$\big(u_\veps,\o_\veps,k_\veps\big)_\veps$.
We thus discover that the triplet $(u,\o,k)$ indeed solves system \eqref{eq:kolm} in the weak sense,
namely
\begin{align*}
&\iint_{\R_+\times\Omega}\left(-\,u\,\d_t\phi_1\,-\,\frac{1}{2}\,u^2\,\d_x\phi_1\,+\,\nu\,\frac{k}{\o}\,\d_xu\,\d_x\phi_1\right)\,dx\,dt\,=\,
\int_\Omega u_{0}\,\phi_1(0,\cdot)\,dx\,, \\
&\iint_{\R_+\times\Omega}\left(-\,\o\,\d_t\phi_2\,-\,\o\,u\,\d_x\phi_2\,-\,\o\,\d_xu\,\phi_2\,+\,
\alpha_1\,\frac{k}{\o}\,\d_x\o\,\d_x\phi_2\right)\,dx\,dt \\
&\qquad\qquad =\,
\int_\Omega \o_{0}\,\phi_2(0,\cdot)\,dx\,-\,\iint_{\R_+\times\Omega}\o^2\,\phi_2\,dx\,dt\,, \\
&\iint_{\R_+\times\Omega}\left(-\,k\,\d_t\phi_3\,-\,k\,u\,\d_x\phi_3\,-\,k\,\d_xu\,\phi_3\,+\,
\alpha_3\,\frac{k}{\o}\,\d_xk\,\d_x\phi_3\right)\,dx\,dt \\
&\qquad\qquad =\,
\int_\Omega k_{0}\,\phi_3(0,\cdot)\,dx\,-\,\iint_{\R_+\times\Omega}k\,\o\,\phi_3\,dx\,dt\,+\,
\alpha_4\iint_{\R_+\times\Omega}\frac{k}{\o}\,\left|\d_xu\right|^2\,\phi_3\,dx\,dt\,,
\end{align*}
for all smooth functions $\phi_{1,2,3}$ belonging to $C^\infty_c(\R_+\times\Omega)$.
In addition, using the regularity properties stated above, see for instance \eqref{ub:u_e}, it is possible to check
that $\big(u,\o,k\big)$ is in fact a strong solution\footnote{As a matter of fact,
by a careful analysis of the viscosity and diffusion terms and a use of the well-known tame estimates, it is possible to
strengthen the condition $\big(u_\veps\big)_\veps\subset W^{1,\infty}_T(H^{-1})$ to the condition
$\big(u_\veps\big)_\veps\subset W^{1,\infty}_T(L^2)$.} to the Kolmogorov system \eqref{eq:kolm}.

To conclude, we remark that, owing to \eqref{est:sol_e} again, there exists a function $\bt\in L^\infty_T(H^2)$ such that
$\sqrt{k_\veps}\,\stackrel{*}{\rightharpoonup}\,\bt$ in that space. However, owing to the pointwise convergence of the $k_\veps$'s, we deduce that
$\bt\,=\,\sqrt{k}$. Studying the equation for the quantities $\sqrt{k_\veps}$ and using an argument similar to the one employed above, we can
get that $\sqrt{k}$ belongs to $C_T(H^{-1})$, hence by interpolation it belongs to $C_T(H^s)$ for any $s<2$. Thanks to this property, \eqref{est:sol_e}
and \eqref{est:d_x^3_e}, we obtain also that the triplet $\big(u,\o,\sqrt{k}\big)$ solves system \eqref{eq:kolm_2}.

\subsection{Uniqueness of solutions} \label{ss:unique}
In order to complete the proof of Theorem \ref{t:local}, it remains us to prove uniqueness of solutions at the previous level of regularity.
First of all, we need the following simple lemma, stating the equivalence of the equations for $k$ and $\sqrt{k}$ in the considered functional framework.
We recall that, for $T>0$, we have defined the functional space
\begin{align*}
\mbb X_T(\Omega)\,&:=\,\Big\{\big(u,\o,k\big)\;\Big|\quad \o,\o^{-1},\,{k}\;\in\,L^\infty\big([0,T]\times\Omega\big)\,,\quad \o>0\,,\quad k\geq0 \\
&\qquad\qquad\qquad\quad
u,\,\o,\,\sqrt{k}\;\in\,C\big([0,T];L^2(\Omega)\big)\,,\quad \d_xu,\,\d_x\o,\,\d_x\sqrt{k}\;\in\,L^\infty\big([0,T]\times\Omega\big) \Big\}\,.
\end{align*}

\begin{lemma} \label{l:k-b}
Let $T>0$ be a given time.
\begin{enumerate}[(a)]
 \item Let $\big(u,\o,k\big)$ be a triplet in $\mbb X_T(\Omega)$ and assume that $k$ solves the third equation in \eqref{eq:kolm}
in the weak sense in $[0,T]\times\Omega$. Then $\bt\,:=\,\sqrt{k}$ is a weak solution to equation \eqref{eq:beta}, related to the initial datum $\bt_0\,:=\,\sqrt{k(0)}$.

\item Let the triplet of functions $\big(u,\o,\bt\big)$ be such that $\big(u,\o,\bt^2\big)\in\mbb X_T(\Omega)$ and 
$\bt$ is a weak solution to \eqref{eq:beta} in $[0,T]\times\Omega$, related to the initial datum $\bt_0$. Set $k\,:=\,\bt^2$. Then $k$ is a weak solution to the
third equation in \eqref{eq:kolm} on $[0,T]\times\Omega$, related to the initial datum $k_0\,:=\,\bt_0^2$.
\end{enumerate}
\end{lemma}

\begin{proof}
The proof is based on simple verifications. To see that statement (b) holds true, it is enough to multiply equation \eqref{eq:beta} by the quantity $2\bt$.
Next, we check that all the integrations by parts are justified in the sense of $\mc D'\big(\,]0,T[\,\times\Omega\big)$, owing to the regularity of the different
quantities. For instance, from equation \eqref{eq:beta}, one gets $\d_t\bt\in L^2_T(H^{-1})$, thus the duality pairing
$$
\langle \d_t\bt\,,\,2\bt\vphi\rangle_{L^2_T(H^{-1})\times L^2_T(H^{1})}
$$
for $\vphi\in\mc D\big(\,]0,T[\,\times\Omega\big)$, is well-justified; from this computation, one obtains that
\begin{equation} \label{eq:distrib}
2\,\bt\,\d_t\bt\,=\,\d_t\bt^2\qquad\qquad \mbox{ in the sense of }\mc D'\big(\,]0,T[\,\times\Omega\big)\,. 
\end{equation}
Notice that the previous equality makes sense, as (owing to the assumption $\big(u,\o,\bt^2\big)\in\mbb X_T(\Omega)$) one has $\sqrt{k}=\bt\in C_T(L^2)$.
In order to see \eqref{eq:distrib}, let us take some $\vphi\in\mc D\big(\,]0,T[\,\times\Omega\big)$, as above. First of all, by a density argument we check that the equality
\[
\bt\,\d_t\vphi\,=\,\d_t\big(\bt\,\vphi\big)\,-\,\vphi\,\d_t\bt
\]
holds true in $L^2_T(H^{-1})$. Next, using that identity, we can compute
\begin{align*}
\lan \d_t\big(\bt^2\big),\vphi\ran_{\mc D'\times\mc D}\,&=\,-\,\lan\bt^2,\d_t\vphi\ran_{\mc D'\times\mc D}\,
=\,-\,\lan\bt,\bt\,\d_t\vphi\ran_{L^2_T(H^1)\times L^2_T(H^{-1})} \\
&=\,-\,\lan\bt,\d_t\big(\bt\,\vphi\big)\ran_{L^2_T(H^1)\times L^2_T(H^{-1})}\,+\,\lan\bt,\vphi\,\d_t\bt\ran_{L^2_T(H^1)\times L^2_T(H^{-1})}\,.
\end{align*}
But $\bt\in W^{1,2}_T(H^{-1})$ and $\bt\,\vphi\in L^2_T(H^1)$, thus a duality argument (proceed \tsl{e.g.} by approximating $\beta$ by smooth functions again) shows that
\[
 -\,\lan\bt,\d_t\big(\bt\,\vphi\big)\ran_{L^2_T(H^1)\times L^2_T(H^{-1})}\,=\,\lan\d_t\bt,\bt\,\vphi\ran_{L^2_T(H^{-1})\times L^2_T(H^{1})}\,,
\]
which in the end proves \eqref{eq:distrib}, as claimed.
One can argue similarly for the terms presenting the space derivatives (using, for instance, an approximation argument as the one used in Proposition 9.4 of \cite{Brezis}).
Finally, for the initial datum, we make use of the fact that $\bt\in C_T(L^2)$ to infer that $\bt^2=k\in C_T(L^1)$.
In the end, statement (b) of the lemma is proved.

Now, let us prove assertion (a). First of all, for $\veps>0$, we define the functions $k_\veps\,:=\,k+\veps$. Then $k_\veps\geq\veps$ in $[0,T]\times\Omega$
and $\big(k_\veps\big)_\veps\subset C_T(L^1)\cap L^\infty_T(H^1)$. Moreover $k_\veps\,\longrightarrow\,k$ for $\veps\ra0^+$, uniformly on $[0,T]\times\Omega$.
Notice also that $k_\veps$ verifies the equation
\[
\d_tk_\veps\,+\,u\,\d_xk_\veps\,-\,\al_3\,\d_x\left(\frac{k_\veps}{\o}\,\d_xk_\veps\right)\,+\,k_\veps\,\o\,=\,\al_4\,\frac{k_\veps}{\o}\,|\d_xu|^2\,-\,
\veps\,\al_3\,\d_x\left(\frac{1}{\o}\,\d_xk_\veps\right)\,-\,\al_4\,\frac{\veps}{\o}\,|\d_xu|^2\,.
\]
Owing to the strictly positive lower bound on $k_\veps$, we can multiply the previous equation by $k_\veps^{-1/2}/2$, thus finding 
(by the same manipulations as in the previous step) an equation for $\sqrt{k_\veps}$. At this point, passing to the limit for $\veps\ra0^+$
is easy, thanks to the convergence properties of $\big(k_\veps\big)_\veps$. This finally implies that $\beta\,=\,\sqrt{k}$ solves \eqref{eq:beta} in the weak sense.
The claimed equality for the initial datum is also satisfied, owing to the time continuity of $\sqrt{k}$ with values in $L^2$.

The lemma is now completely proven.
\end{proof}

With this lemma at hand, we can prove a $L^2$ stability estimate for strong solutions of system \eqref{eq:kolm_2} possessing vanishing mean turbulent kinetic energy.
This is provided by the next statement. The uniqueness of solutions claimed in Theorem \ref{t:local} will be a direct consequence of it.

\begin{thm} \label{th:stab}
Let the triplets $\big(u_1,\o_1,k_1\big)$ and $\big(u_2,\o_2,k_2\big)$ be two solutions to the Kolmogorov two-equation model \eqref{eq:kolm}.
Assume that, for some time $T>0$, the following hypotheses are satisfied for any $j\in\{1,2\}$:
\begin{enumerate}[(i)]
 \item 
 one has $\o_j>0$ and $k_j\geq0$ on $[0,T]\times\Omega$;
 \item 
 the quantities $\o_j,$ $\o_j^{-1}$ and $\sqrt{k_j}$ belong to $L^\infty\big([0,T];L^\infty(\Omega)\big)$;
 \item 
 the functions $\d_xu_j$, $\d_x\o_j$ and $\d_x\sqrt{k_j}$ all belong to $L^\infty\big([0,T];L^\infty(\Omega)\big)$.
\end{enumerate}
After defining the quantities
$$
U\,:=\,u_1-u_2\,,\qquad \Sigma\,:=\,\omega_1-\omega_2\qquad\mbox{ and }\qquad B\,:=\,\sqrt{k_1}-\sqrt{k_2}\,,
$$
assume also that $U$, $\Sigma$ and $B$ all belong to $C\big([0,T];L^2(\Omega)\big)$. 

Then, there exist a positive constant $C\,=\, C(\nu,\al_1\ldots\al_4)$ and a function $\Theta\in L^1\big([0,T]\big)$ such that, for all times $t\in[0,T]$,
one has
\[
\left\|U(t)\right\|_{L^2}^2+\left\|\Sigma(t)\right\|_{L^2}^2+\left\|B(t)\right\|_{L^2}^2\,\leq\,\Big(\left\|U(0)\right\|_{L^2}^2+\left\|\Sigma(0)\right\|_{L^2}^2+\left\|B(0)\right\|_{L^2}^2\Big)\,\exp\left(C\int^t_0\Theta(\t)\,d\t\right)\,.
\]
\end{thm}

\begin{proof}
Before starting the computations, let us introduce some more notation and define
\begin{equation} \label{eq:K-B} 
K\,:=\,k_1-k_2\,=\,B\,\big(\bt_1+\bt_2\big)\,,
\end{equation} 
where, as in Section \ref{s:a-priori}, we have set $\bt_i\,:=\,\sqrt{k_i}$ for $i\in\{1,2\}$.
However, for lightening the notation, we will keep using the symbols $k_{1,2}$ and $K$ when convenient.

Thanks to Lemma \ref{l:k-b}, we know that, for any $j\in\{1,2\}$, the triplet $\big(u_j,\o_j,\bt_j\big)$ solves system \eqref{eq:kolm_2}. Thus,
we can perform the stability estimates on system \eqref{eq:kolm_2}. We point out that, in view of assumption (iii), the three functions 
$U$, $\Sigma$ and $B$ belong in particular to $L^2\big([0,T];H^1(\Omega)\big)$, with $\d_tU$, $\d_t\Sigma$ and $\d_tB$ in $L^2\big([0,T];H^{-1}(\Omega)\big)$. Then, our energy estimates below are rigorously justified, in light of the regularity assumptions formulated in the statement of the theorem: this is a consequence of Theorem 3 in Section 5.9 of \cite{Evans}, a standard regularisation-in-space process (which we omit here) and related commutator estimates.

\medbreak
To begin with, we see that $U$ solves the equation
\begin{align} \label{eq:U}
\d_tU\,+\,u_1\,\d_xU\,-\,\nu\,\d_x\left(\frac{k_1}{\o_1}\,\d_xU\right)\,=\,-\,U\,\d_xu_2\,+\,\nu\,\d_x\big(\Delta\,\d_xu_2\big)\,,
\end{align}
where we have defined
\begin{equation} \label{eq:D}
\Delta\,:=\,\dfrac{k_1}{\o_1}-\dfrac{k_2}{\o_2}\,=\,-\,\frac{k_1}{\o_1\,\o_2}\,\Sigma\,+\,\frac{1}{\o_2}\,K\,.
\end{equation}
A simple energy estimate for \eqref{eq:U} yields
\begin{align}
\dfrac{1}{2}\,\dfrac{d}{dt}\|U\|_{L^2}^2\,+\,\nu\int_\Omega\frac{k_1}{\o_1}\,\left|\d_xU\right|^2\,dx\,&=\,\int_\Omega\left(\frac{1}{2}\,\d_xu_1\,-\,\d_xu_2\right)|U|^2\,dx \label{est:U_1} \\
&\; +\,
\nu\int_\Omega \frac{k_1}{\o_1\,\o_2}\,\Sigma\,\d_xu_2\,\d_xU\,dx\,+\,\nu\int_\Omega\frac{1}{\o_2}\,K\,\d_xu_2\,\d_xU\,dx\,. \nonumber
\end{align}
It is apparent that the problem relies in the control of the last term on the right-hand side of the previous equality.
Indeed, using the decomposition \eqref{eq:K-B} forces us to deal with an integral of the following form:
\begin{equation} \label{int:to-control}
\nu\int_\Omega\frac{1}{\o_2}\,B\,\beta_2\,\d_xu_2\,\d_xU\,dx\,.
\end{equation}
Now, in order to control the $\d_xU$ term in $L^2$ (without resorting to higher order estimates), one needs to make a factor $\beta_1$ appear
in front of it, thus finding
\[
\nu\int_\Omega\frac{1}{\o_2}\,B\,\beta_2\,\d_xu_2\,\d_xU\,dx\,=\,
\nu\int_\Omega\frac{1}{\o_2}\,B\,\frac{\beta_2}{\beta_1}\,\d_xu_2\,\beta_1\d_xU\,dx\,.
\]
Of course, one disposes of no control on the quantity $\beta_2/\beta_1$. Also imposing some compatibility conditions on the vanishing
of both $\beta_1$ and $\beta_2$ does not look good, as tracking the propagation of those compatibility conditions seems to be out of reach in our theory.
Thus, we decide to follow a different strategy,
which however requires to impose the strong conditions (iii) of the statement.

More precisely, we remark that $U$ solves not only \eqref{eq:U}, but also the equivalent equation
\[
\d_tU\,+\,u_2\,\d_xU\,-\,\nu\,\d_x\left(\frac{k_2}{\o_2}\,\d_xU\right)\,=\,-\,U\,\d_xu_1\,+\,\nu\,\d_x\big(\Delta\,\d_xu_1\big)\,.
\]
If we now perform an energy estimate on this relation, and we sum up the resulting expression with \eqref{est:U_1}, we find
\begin{align}
&\dfrac{d}{dt}\|U\|_{L^2}^2\,+\,\nu\int_\Omega\left(\frac{\bt^2_1}{\o_1}\,+\,\frac{\bt_2^2}{\o_2}\right)\,\left|\d_xU\right|^2\,dx \label{est:U_2} \\
&\leq 2\,\big(\left\|\d_xu_1\right\|_{L^\infty}+\left\|\d_xu_2\right\|_{L^\infty}\big)\,\|U\|_{L^2}^2 \nonumber \\
&\; +\,\nu\,\left|\int_\Omega \frac{\bt^2_1}{\o_1\,\o_2}\,\Sigma\,\big(\d_xu_2+\d_xu_1\big)\,\d_xU\,dx\right| 
\,+\,\nu\,\left|\int_\Omega\frac{1}{\o_2}\,B\,\big(\bt_1+\bt_2\big)\,\big(\d_xu_2+\d_xu_1\big)\,\d_xU\,dx\right|\,. \nonumber
\end{align}
where, now, we have used \eqref{eq:K-B} for expressing $K$ and we have replaced $k_j$ by $\bt_j^2$.

At this point, we easily estimate
\begin{align*}
&\nu\left|\int_\Omega \frac{\bt^2_1}{\o_1\,\o_2}\,\Sigma\,\big(\d_xu_2+\d_xu_1\big)\,\d_xU\,dx\right| \\
&\qquad \leq\,\nu\,\big(\left\|\d_xu_1\right\|_{L^\infty}+\left\|\d_xu_2\right\|_{L^\infty}\big)\,\left\|\frac{\bt_1}{\sqrt{\o_1}\,\o_2}\right\|_{L^\infty}\,\|\Sigma\|_{L^2}\,\left\|\frac{\bt_1}{\sqrt{\o_1}}\,\d_xU\right\|_{L^2} \\
&\qquad \leq\,\de\,\nu\int_\Omega \frac{\bt_1^2}{\o_1}\,\left|\d_xU\right|^2\,dx\,+\,C(\de,\nu)\,\big(\left\|\d_xu_1\right\|^2_{L^\infty}+\left\|\d_xu_2\right\|^2_{L^\infty}\big)\,\left\|\frac{\bt_1}{\sqrt{\o_1}\,\o_2}\right\|^2_{L^\infty}\,\|\Sigma\|^2_{L^2}\,,
\end{align*}
where $\de>0$ can be taken arbitrarily small and, like in Section \ref{s:a-priori}, the constant $C(\de,\nu)>0$ depends only on the quantities in the brackets.

Similarly, for $j\in\{1,2\}$, we can bound
\begin{align*}
& \nu\left|\int_\Omega\frac{1}{\o_2}\,B\,\bt_j\,\big(\d_xu_2+\d_xu_1\big)\,\d_xU\,dx\right| \\
&\qquad \leq\,\nu\,\big(\left\|\d_xu_1\right\|_{L^\infty}+\left\|\d_xu_2\right\|_{L^\infty}\big)\,\left\|\frac{\sqrt{\o_j}}{\o_2}\right\|_{L^\infty}\,\|B\|_{L^2}\,\left\|\frac{\bt_j}{\sqrt{\o_j}}\,\d_xU\right\|_{L^2} \\
&\qquad \leq\,\de\,\nu\int_\Omega \frac{\bt_j^2}{\o_j}\,\left|\d_xU\right|^2\,dx\,+\,C(\de,\nu)\,\big(\left\|\d_xu_1\right\|^2_{L^\infty}+\left\|\d_xu_2\right\|^2_{L^\infty}\big)\,\left\|\frac{\sqrt{\o_j}}{\o_2}\right\|^2_{L^\infty}\,\|B\|^2_{L^2}\,,
\end{align*}
where $\de>0$ is again as small as one wants. Therefore, taking $\de$ small enough and inserting those bounds into \eqref{est:U_2}, we finally find
\begin{align}
\dfrac{d}{dt}\|U\|_{L^2}^2\,+\,\nu\int_\Omega\left(\frac{\bt^2_1}{\o_1}\,+\,\frac{\bt_2^2}{\o_2}\right)\,\left|\d_xU\right|^2\,dx\,\lesssim\,\Theta_1(t)\,\mbb{E}(t)\,,
\label{est:U_fin}
\end{align}
where the implicit multiplicative constant depends only on $\nu>0$, and where we have defined the energy function
\[
\mbb E(t)\,:=\,\left\|U\right\|_{L^2}^2\,+\,\left\|\Sigma\right\|_{L^2}^2\,+\,\left\|B\right\|_{L^2}^2
\]
and the function
\[
\Theta_1(t)\,:=\,\left\|\left(\d_xu_1,\d_xu_2\right)\right\|_{L^\infty}\,+\,\left\|\left(\d_xu_1,\d_xu_2\right)\right\|_{L^\infty}^2\left(
\left\|\frac{\bt_1}{\sqrt{\o_1}\,\o_2}\right\|^2_{L^\infty}\,+\,\left\|\frac{\sqrt{\o_1}}{\o_2}\right\|^2_{L^\infty}\,+\,\left\|\frac{1}{\o_2}\right\|_{L^\infty}\right)\,.
\]

Next, we consider the quantity $\Sigma\,:=\,\o_1-\o_2$. Easy computations yield that $\Sigma$ satisfies equivalently the two equations
\begin{align*}
\d_t\Sigma\,+\,u_1\,\d_x\Sigma\,-\,\al_1\,\d_x\left(\frac{k_1}{\o_1}\,\d_x\Sigma\right)\,+\,\al_2\,\big(\o_1+\o_2\big)\,\Sigma\,&=\,-\,U\,\d_x\o_2\,+\,\al_1\,\d_x\big(\Delta\,\d_x\o_2\big) \\
\d_t\Sigma\,+\,u_2\,\d_x\Sigma\,-\,\al_1\,\d_x\left(\frac{k_2}{\o_2}\,\d_x\Sigma\right)\,+\,\al_2\,\big(\o_1+\o_2\big)\,\Sigma\,&=\,-\,U\,\d_x\o_1\,+\,\al_1\,\d_x\big(\Delta\,\d_x\o_1\big)\,.
\end{align*}
An energy estimate on both equations, as before, thus gives us
\begin{align*}
&\hspace{-0.2cm}\frac{d}{dt}\left\|\Sigma\right\|^2_{L^2}\,+\,\al_1\int_\Omega\left(\frac{\bt^2_1}{\o_1}\,+\,\frac{\bt_2^2}{\o_2}\right)\,\left|\d_x\Sigma\right|^2\,dx \\
&+\,2\,\al_2\int_\Omega\left(\o_1+\o_2\right)\,|\Sigma|^2\,dx\,\leq\,\frac{1}{2}\,\left\|\left(\d_xu_1,\d_xu_2\right)\right\|_{L^\infty}\,\|\Sigma\|_{L^2}^2 \\
&\qquad\qquad\qquad\qquad\,+\,
\left\|\left(\d_x\o_1,\d_x\o_2\right)\right\|_{L^\infty}\,\|\Sigma\|_{L^2}\,\|U\|_{L^2}\,+\,\al_1\left|\int_\Omega\left(\d_x\o_1+\d_x\o_2\right)\,\Delta\,\d_x\Sigma\,dx\right|\,.
\end{align*}
Using the expression of $\Delta$ found in \eqref{eq:D}, it is easy to see (repeating the computations above) that, for any $\de>0$ arbitrarily fixed, we can bound
\begin{align*}
&\al_1\left|\int_\Omega\left(\d_x\o_1+\d_x\o_2\right)\,\Delta\,\d_x\Sigma\,dx\right| \\
&\leq\,\de\,\al_1\int_\Omega \frac{\bt_1^2}{\o_1}\,\left|\d_x\Sigma\right|^2\,dx\,+\,C\,\left\|\left(\d_x\o_1,\d_x\o_2\right)\right\|^2_{L^\infty}\,\left\|\frac{\bt_1}{\sqrt{\o_1}\,\o_2}\right\|^2_{L^\infty}\,\|\Sigma\|^2_{L^2} \\
&+\,\de\,\al_1\int_\Omega\left(\frac{\bt_1^2}{\o_1}\,+\,\frac{\bt_2^2}{\o_2}\right)\,\left|\d_x\Sigma\right|^2\,dx\,+\,C\,\left\|\left(\d_x\o_1,\d_x\o_2\right)\right\|^2_{L^\infty}\left(
\left\|\frac{\sqrt{\o_1}}{\o_2}\right\|^2_{L^\infty}\,+\,\left\|\frac{1}{\o_2}\right\|_{L^\infty}\right)\,\|B\|_{L^2}^2\,,
\end{align*}
for a suitable constant $C=C(\de,\al_1)>0$.
This estimate immediately implies that, for a (implicit) positive multiplicative constant depending only on $\alpha_1$, one has
\begin{equation} \label{est:Sigma}
\frac{d}{dt}\left\|\Sigma\right\|^2_{L^2}\,+\,\al_1\int_\Omega\left(\frac{\bt^2_1}{\o_1}\,+\,\frac{\bt_2^2}{\o_2}\right)\,\left|\d_x\Sigma\right|^2\,dx\,+\,
\al_2\int_\Omega\left(\o_1+\o_2\right)\,|\Sigma|^2\,dx\,\lesssim\,\Theta_2(t)\,\mbb E(t)\,,
\end{equation}
where the function $\Theta_2$ is analogous to $\Theta_1$ and is defined by
\[
\Theta_2(t)\,:=\,\left\|\left(\d_xu_1,\d_xu_2\right)\right\|_{L^\infty}\,+\,\left\|\left(\d_x\o_1,\d_x\o_2\right)\right\|_{L^\infty}^2\left(
\left\|\frac{\bt_1}{\sqrt{\o_1}\,\o_2}\right\|^2_{L^\infty}\,+\,\left\|\frac{\sqrt{\o_1}}{\o_2}\right\|^2_{L^\infty}\,+\,\left\|\frac{1}{\o_2}\right\|_{L^\infty}\right)\,.
\]

The last step is to write an equation for $B\,:=\,\bt_1-\bt_2$ and perform an energy estimate for it. First of all, after setting
\[
\wtilde\Delta\,:=\,\frac{\bt_1}{\o_1}\,-\,\frac{\bt_2}{\o_2}\,=\,-\,\frac{\bt_1}{\o_1\,\o_2}\,\Sigma\,+\,\frac{1}{\o_2}\,B\,,
\]
in analogy with $\Delta$ in \eqref{eq:D}, from the last equation of \eqref{eq:kolm_2} we see that $B$ solves
\begin{align*}
\d_tB\,+\,u_1\,\d_xB\,-\,\al_3\,\d_x\left(\frac{k_1}{\o_1}\,\d_xB\right)\,+\,\frac{1}{2}\,\o_1\,B\,&=\,-\,U\,\d_x\bt_2\,+\,\al_3\,\d_x\big(\Delta\,\d_x\bt_2\big)\,-\,\frac{1}{2}\,\Sigma\,\bt_2 \\
&\quad +\,\frac{\al_4}{2}\,\wtilde\Delta\,|\d_xu_2|^2\,+\,\frac{\al_4}{2}\,\frac{\bt_1}{\o_1}\,\d_xU\,\big(\d_xu_1+\d_xu_2\big) \\
&\quad +\,2\,\al_3\,\wtilde\Delta\,|\d_x\bt_2|^2\,+\,2\,\al_3\,\frac{\bt_1}{\o_1}\,\d_xB\,\big(\d_x\bt_1+\d_x\bt_2\big)\,.
\end{align*}
We now test the previous relation on $B$ and integrate by parts whenever useful. First of all, we notice that, in order to handle the $\Delta$ term,
we need to resort again to a symmetrization of the previous equation. Arguing as for estimating the corresponding term in the equation for $U$, we see that
\begin{align*}
&\al_3\left|\int_\Omega\left(\d_x\bt_1+\d_x\bt_2\right)\,\Delta\,\d_xB\,dx\right|\,=\,\al_3\left|\int_\Omega\left(\d_x\bt_1+\d_x\bt_2\right)\left(-\,\frac{k_1}{\o_1\,\o_2}\,\Sigma\,+\,\frac{1}{\o_2}\,B\,\big(\bt_1+\bt_2\big)\right)\d_xB\,dx\right|
\end{align*}
can be controlled, for $\de>0$ arbitrarily small, by the quantity
\begin{align*}
&\de\,\al_3\int_\Omega \frac{\bt_1^2}{\o_1}\,\left|\d_xB\right|^2\,dx\,+\,C\,\left\|\left(\d_x\bt_1,\d_x\bt_2\right)\right\|^2_{L^\infty}\,\left\|\frac{\bt_1}{\sqrt{\o_1}\,\o_2}\right\|^2_{L^\infty}\,\|\Sigma\|^2_{L^2} \\
&\quad+\,\de\,\al_3\int_\Omega\left(\frac{\bt_1^2}{\o_1}\,+\,\frac{\bt_2^2}{\o_2}\right)\,\left|\d_xB\right|^2\,dx\,+\,C\,\left\|\left(\d_x\bt_1,\d_x\bt_2\right)\right\|^2_{L^\infty}\,\left(
\left\|\frac{\sqrt{\o_1}}{\o_2}\right\|^2_{L^\infty}\,+\,\left\|\frac{1}{\o_2}\right\|_{L^\infty}\right)\,\|B\|_{L^2}^2\,,
\end{align*}
for a suitable constant $C=C(\de,\al_3)>0$.
Next, we have to control all the other terms appearing in the equation for $B$. Recall that we have symmetrized the equations: then, for $j\in\{1,2\}$, we start to bound
\[
\left|\int_\Omega\left(-u_j\,\d_xB\,-\,U\,\d_x\bt_j\,-\,\frac{1}{2}\,\Sigma\,\bt_j\right)\,B\,dx\right|\,\lesssim\,\mbb E\,\left(\|\d_xu_j\|_{L^\infty}+\|\d_x\bt_j\|_{L^\infty}+\|\bt_j\|_{L^\infty}\right)\,.
\]
Next, using the definition of $\wtilde\Delta$, we easily obtain
\begin{align*}
\left|\int_\Omega\left(\frac{\al_4}{2}\,\wtilde\Delta\,|\d_xu_j|^2\,+\,2\,\al_3\,\wtilde\Delta\,|\d_x\bt_j|^2\right)B\,dx\right|\,\lesssim\,
\mbb E\,\left\|\big(\d_xu_j,\d_x\bt_j\big)\right\|_{L^\infty}^2\,\left(\left\|\frac{\bt_1}{\o_1\,\o_2}\right\|_{L^\infty}\,+\,\left\|\frac{1}{\o_2}\right\|_{L^\infty}\right)\,,
\end{align*}
where the (implicit) multiplicative constant depends on the parameters $\al_3$ and $\al_4$.
Finally, for $\de>0$ as above, we can estimate
\begin{align*}
&\al_3\,\left|\int_\Omega \frac{\bt_j}{\o_j}\,\d_xB\,\big(\d_x\bt_1+\d_x\bt_2\big)\,B\,dx\right| \\
&\qquad\qquad\qquad\leq\,\de\,\al_3\int_\Omega\frac{\bt_j^2}{\o_j}\left|\d_xB\right|^2\,dx\,+\,
C(\de,\al_3)\,\left\|\frac{1}{\o_j}\right\|_{L^\infty}\,\left\|\big(\d_x\bt_1,\d_x\bt_2\big)\right\|_{L^\infty}^2\,\|B\|_{L^2}^2\,,
\end{align*}
whereas, for any $\eta>0$ arbitrarily small, we have
\begin{align*}
&\al_4\,\left|\int_\Omega \frac{\bt_j}{\o_j}\,\d_xU\,\big(\d_xu_1+\d_xu_2\big)\,B\,dx\right| \\
&\qquad\qquad\qquad\leq\,\eta\,\nu\int_\Omega\frac{\bt_j^2}{\o_j}\left|\d_xU\right|^2\,dx\,+\,
C(\eta,\nu,\al_4)\,\left\|\frac{1}{\o_j}\right\|_{L^\infty}\,\left\|\big(\d_xu_1,\d_xu_2\big)\right\|_{L^\infty}^2\,\|B\|_{L^2}^2\,.
\end{align*}
In the end, if we choose $\de>0$ small enough and we fix (say) $\eta=1/4$, putting all the previous estimates together yields,
for a suitable multiplicative constant $C=C\big(\nu,\al_1\ldots\al_4\big)>0$, the inequality
\begin{align}
&\frac{d}{dt}\left\|B\right\|^2_{L^2}\,+\,\al_3\int_\Omega\left(\frac{\bt^2_1}{\o_1}\,+\,\frac{\bt_2^2}{\o_2}\right)\,\left|\d_xB\right|^2\,dx\,+\,
\int_\Omega\left(\o_1+\o_2\right)\,|B|^2\,dx \label{est:B} \\
&\qquad\qquad\qquad\qquad\qquad\qquad\qquad\qquad
\leq\,C\,\Theta_3(t)\,\mbb E(t)\,+\,\frac{\nu}{4}\int_\Omega\left(\frac{\bt^2_1}{\o_1}\,+\,\frac{\bt_2^2}{\o_2}\right)\,\left|\d_xU\right|^2\,dx\,,
\nonumber
\end{align}
where we have defined the function
\begin{align*}
\Theta_3(t)\,&:=\,\left\|\left(\d_xu_1,\d_xu_2,\d_x\bt_1,\d_x\bt_2\right)\right\|_{L^\infty}\,+\,\left\|\big(\bt_1,\bt_2\big)\right\|_{L^\infty} \\
&\qquad +\,\left\|\left(\d_xu_1,\d_xu_2,\d_x\bt_1,\d_x\bt_2\right)\right\|_{L^\infty}^2 \\
&\qquad\qquad\qquad\qquad \times\left(
\left\|\frac{\bt_1}{\sqrt{\o_1}\,\o_2}\right\|^2_{L^\infty}\,+\,\left\|\frac{\sqrt{\o_1}}{\o_2}\right\|^2_{L^\infty}\,+\,\left\|\left(\frac{1}{\o_1},\frac{1}{\o_2}\right)\right\|_{L^\infty}\,+\,\left\|\frac{\bt_1}{\o_1\,\o_2}\right\|_{L^\infty}\right)\,.
\end{align*}

We are now in the position of closing our stability estimates. Indeed, summing up \eqref{est:U_fin}, \eqref{est:Sigma} and \eqref{est:B}, after absorbing
the last term  on the right-hand side of \eqref{est:B} into the left-hand side, we finally get
\[
\frac{d}{dt}\mbb E(t)\,\lesssim\,\Big(\Theta_1(t)\,+\,\Theta_2(t)\,+\,\Theta_3(t)\Big)\,\mbb E(t)\,,
\]
and an application of the Gr\"onwall lemma concludes the proof.
\end{proof}

We end this section by formulating a couple of remarks.
\begin{rmk} \label{r:unique}
In light of estimates \eqref{est:U_fin}, \eqref{est:Sigma} and \eqref{est:B}, and of the definition of the functions $\Theta_{1,2,3}$, one could
relax the $L^\infty_T$ assumptions formulated in items (ii) and (iii) of Theorem \ref{th:stab}. This would lead to uniqueness of solutions in a class
which is actually larger than the one claimed in Theorem \ref{t:local}.

However, the conditions to require would be much more in number, and more complex to formulate. For this reason, we have decided to keep simple assumptions, and to
limit ourselves to the present form of Theorem \ref{th:stab}.
\end{rmk}

\begin{rmk} \label{r:weak-strong}
Observe that Theorem \ref{th:stab} is \emph{not} a weak-strong uniqueness type result, because assumption (iii) has to be made on both triplets of solutions.
This issue has been discussed in the course of the proof, see in particular formula \eqref{int:to-control} and the comments below.
We point out that the problem comes from the possible vanishing of \emph{both} $k_1$ and $k_2$.

In passing, we remark that weak-strong uniqueness results have known great and important developments in the context of fluid mechanics models in the last few years,
see \tsl{e.g.} \cite{F-Jin-N} for a general statement. 
However, in all those results the density function related to the regular solution is always assumed to be bounded \emph{away from vacuum}. In our case, instead,
both $k_1$ and $k_2$ may vanish, as we already remarked above. On the other hand, very likely it might possible to avoid
the additional assumption (iii), thus coverting our stability result into a weak-strong uniqueness statement,
in the case one of those two quantities remains bounded far from $0$.
\end{rmk}

\subsection{The blow-up mechanism} \label{ss:blow-up}

In this section, we prove Theorem \ref{t:blow-up} and show that, in general, the solutions constructed in the previous part blow up in finite time.
More precisely, we are able to prove that, for initial data which possess a suitable symmetry, the corresponding solutions to the Kolmogorov system \eqref{eq:kolm}
exhibit a growth in the space derivative of the velocity field. The precise statement is the following one.

\begin{thm} \label{t:symmetry}
Let $\big(u_0,\o_0,k_0\big)$ be a triplet of functions belonging to $H^3(\Omega)$, with $\omega_*\leq\o_0\leq\o^*$, for some strictly
positive constants $0<\o_*\leq\o^*$, and with $k_0\geq0$ such that $\sqrt{k_0}\in H^3(\Omega)$.
Assume the following conditions:
\begin{enumerate}[(i)]
 \item $k_0(0)=0$;
 \item $u_0$ is odd with respect to $0$, while $\o_0$ and $k_0$ are even with respect to $0$;
 \item $\d_xu_0(0)<0$.
\end{enumerate}

Then, the corresponding $H^3$ solution $\big(u,\o,k\big)$, whose existence is guaranteed by Theorem \ref{t:local}, blows up in finite time.
More precisely, if the solution does not blow up first at a different place, there exists a time $t_0>0$ such that
\[
\lim_{t\ra t_0^-}\d_xu(t,0)\,=\,-\infty\,.
\]
\end{thm}

\begin{proof}
We are going to show that the solution $\big(u,\o,k\big)$ cannot be global arguing by contradiction: more precisely, we will prove that, assuming that the solution exists and remains smooth, then there will be a blow-up of the solution in a finite time $t=t_0$ exactly at $x=0$.

Recall that we have supposed $u_0$ to be odd, and $k_0$ and $\o_0$ to be even with respect to the origin. As a consequence, besides the assumptions
of Theorem \ref{t:symmetry}, we also know that
\[
\o_0(0)\,>\,0\qquad\qquad\mbox{ and }\qquad\qquad \d_xk_0(0)\,=\,\d_x\o_0(0)\,=\,0\,.
\]
In addition, we notice that the equations preserve the previous symmetries. To see this,
let us define the triplet $\big(\wtilde u, \wtilde \o, \wtilde k\big)$ as
\[
\wtilde u(t,x)\,:=\,-\,u(t,-x)\,,\qquad \wtilde \o(t,x)\,:=\,\o(t,-x)\qquad \mbox{ and }\qquad
\wtilde k(t,x)\,:=\,k(t,-x)\,.
\]
Simple computations show that $\big(\wtilde u, \wtilde \o, \wtilde k\big)$ is still a solution of the Kolmogorov system \eqref{eq:kolm},
related to the same initial datum $\big(u_0,\o_0,k_0\big)$. 
Therefore, by uniqueness of solutions, we deduce that
$\big(\wtilde u, \wtilde \o, \wtilde k\big)\,\equiv\,\big(u,\o,k\big)$, which in turn implies
\[
\forall\,t\,\geq\,0\,,\qquad\qquad\qquad u(t,\cdot) \quad \mbox{ is odd }\,,\qquad\qquad \o(t,\cdot)\,,\quad k(t,\cdot)\quad \mbox{ are even}\,.
\]

Since the functions $(u,k,\o)$ are smooth in space, we can repeat the ODE analysis performed in Subsection \ref{ss:ode}.
Thus, after defining the function
\[
a(t)\,:=\,\al_3\,\frac{\d_x^2k(t,0)}{\o(t,0)}\,-\,\o(t,0)\,+\,\al_4\,\frac{\left|\d_xu(t,0)\right|^2}{\o(t,0)}
\]
and using the assumption $k_0(0)=0$ and the fact that $\partial_x \omega(t,0)=\partial_x k(t,0)=0$ (which follows from the spacial symmetry
of the functions $\o$ and $k$), we easily find that
$$
\frac{d}{dt}k(t,0)\,=\,a(t)\,k(t,0)\qquad\qquad\Longrightarrow\qquad\qquad \forall\,t\geq0\,,\qquad k(t,0)\,=\,0\,.
$$
Therefore, at any time $t\geq0$, the point $x=0$ remains a point of minimum for the function $k(t)$.

Now, we differentiate the equation for $u$ with respect to $x$ and compute the resulting expression at the point $x=0$. 
By virtue of the previous properties and the fact that $\partial_x^2 u(t,0)=0$ by symmetry, we find that $\xi(t)\,:=\,\d_xu(t,0)$ solves
\begin{equation} \label{eq:ode-gamma}
\frac{d}{dt}\xi(t)\,=\,-\,\xi^2(t)\,+\,b(t)\,\xi(t)\,,\qquad\qquad \mbox{ with }\qquad b(t)\,:=\,\nu\,\frac{\d_x^2k(t,0)}{\omega(t,0)}\,,
\end{equation}
related to the initial datum $\xi_{|t=0}\,=\,\xi_0\,:=\,\d_xu_0(0)$.
Notice that $b(t)\geq0$ for all times, since $x=0$ is a point of minimum of $k(t)$, for all $t\ge0$.

Observe that, by assumption, one has $\xi_0<0$, so $d\xi/dt$ starts negative; by \eqref{eq:ode-gamma}, $\xi$ is decreasing,
thus it stays negative for all times. Then one deduces that
$$
\frac{d}{dt}\xi\,\leq\,-\xi^2\qquad\Longrightarrow\qquad \xi(t)\,\leq\,\frac{\d_xu_0(0)}{1\,+\,\d_xu_0(0)\,t}\,=\,-\,\frac{\left|\d_xu_0(0)\right|}{1\,-\,\left|\d_xu_0(0)\right|\,t}\,,
$$
where the last equality holds, owing to the fact that $\d_xu_0(0)<0$.
The previous inequality obviously implies blow-up in finite time for the quantity $\xi(t)=\d_xu(t,0)$.

Theorem \ref{t:symmetry} is thus completely proved, and so does Theorem \ref{t:blow-up}.
\end{proof}

\appendix

\section{Appendix: well-posedness of a porus medium equation} \label{app:PME}
The goal of the present appendix is to discuss the well-posedness of the porum medium type equation \eqref{intro_eq:PME}, which
we rewrite here in a slightly more general form.

Let $\theta=\theta(t)$ and $\rho=\rho(t)$ be two positive scalar functions depending only on the time variable $t\in\R_+$, with the property that
\begin{equation} \label{assump:coeff}
\forall\,t\geq0\,,\qquad\qquad \theta(t)>0 \qquad \mbox{ and }\qquad \rho(t)>0\,.
\end{equation}
Consider the following initial-value problem, set on $\R_+\times\Omega$, where $\Omega=\T$ is the one-dimensional torus as above:
\begin{equation} \label{eq:PME}
\left\{ \begin{array}{l}
         \d_tk\,-\,\theta(t)\,\dd\left(k^2\right)\,=\,-\,\rho(t)\,k \\[1ex]
         k_{|t=0}\,=\,k_0\,.
        \end{array}
\right.
\end{equation}

Here, again the initial datum $k_0\geq0$ will be taken a positive function. We will assume $k_0\in H^1(\Omega)$.
Our goal is to show that this equation is well-posed in $H^1$ without requiring
any special assumption on the support of $k_0$. In particular, the regularity requirement over $\sqrt{k_0}$ can be completely dropped here.
This hints at the fact that formulating such an assumption is really specific to the non-linear coupling of the unknowns inside the diffusion operator, namely
of $u$ and $k$ in system \eqref{eq:kolm}, or of $u$ and $\g$ in the toy-model \eqref{eq:toy}.

One may argue that the trick relies on the use of $H^1$ regularity estimates, instead of $H^2$ estimates as we use in our paper.
Actually, this is not true: let us explain better this point in the next remark, where for simplicity we deal only with the toy-model \eqref{eq:toy}.

\begin{rmk} \label{r:H^1}
Observe that, when performing $H^1$ estimates on the toy model \eqref{eq:toy}, one finds a new bad term, simpler but similar to \eqref{eq:bad}, namely
\[
 \int_\Omega\d_x\g\,\d_xu\,\dd u\,dx\,=\,\frac{1}{2}\,\int_\Omega\d_x\g\,\d_x\left((\d_xu)^2\right)\,dx\,.
\]
It is easy to see that this term can be controlled only in two ways: either by assuming a control on $\sqrt{\g}$, as pursued in this paper,
or (after an integration by parts) by looking for, at least, a $L^2$ bound on the second order derivative $\dd\g$.
Now, the non-linear coupling of the equation for $\g$ forces one to propagate also the $H^2$ norm of $u$, thus coming back to the problem discussed in
\eqref{eq:bad}.

This consideration shows somehow the ``necessity'' of working with $\sqrt{\g}$ instead of $\g$ in the toy-model \eqref{eq:toy}, and with $\sqrt{k}$
instead of $k$ for the full Kolmogorov system \eqref{eq:kolm}.
\end{rmk}

Next, let us comment on the propagation of the $H^2$ norm for the prous medium equation \eqref{eq:PME}.
Direct computations show that one cannot control the $L^2$ norm of $\dd k$ uniformly in time, as this quantity satisfies
an inequality of the form
\begin{align*}
 \frac{d}{dt}\left\|\d_x^2k\right\|^2_{L^2}\,+\,\rho(t)\,\left\|\dd k\right\|^2_{L^2}\,+\,\theta(t)\int_\Omega k\,\left|\d_x^3k\right|^2\,dx\,&\lesssim\,
\theta(t)\,\left\|\d_xk\right\|_{L^\infty}\,\left\|\dd k\right\|^{2}_{L^2} \\
&\lesssim\,\theta(t)\,\left\|\dd k\right\|^{2+\delta}_{L^2}\,,
\end{align*}
for a suitable $\de>0$,
and (if one wants to stick to \eqref{intro_eq:PME}, coming from the Kolmogorov system)
one only has $\rho(t)\approx(1+t)^{-1}$, whereas $\theta(t)\approx(1+t)$.
From this perspective, one cannot hope for a global propagation of $H^2$ regularity for equation \eqref{eq:PME}, in general.
Even worse, paper \cite{F-GB_ZAMP} shows blow-up results of smooth solutions for very similar equations.

After those preliminary remarks, we are now ready to state the main result of this appendix.
We stress the fact that (as it will appear clear from our proof) this result is purely one-dimensional.

\begin{thm} \label{t:PME}
Let the scalar functions $\theta=\theta(t)$ and $\rho=\rho(t)$ satisfy assumption \eqref{assump:coeff}. Let $k_0\geq0$ be a positive scalar
function such that $k_0\in H^1(\Omega)$.

Then, there exists a unique global in time (weak) solution $k$ to the initial value problem \eqref{eq:PME}, such that $k\geq0$ and
\[
k\,\in\, L^\infty\big(\R_+;H^1(\Omega)\big)\,\cap\,\bigcap_{s<1}C\big(\R_+;H^s(\Omega)\big)\,,
\]
together with $\sqrt{k}\,\d_x^2k\,\in\,L^2_{\rm loc}\big(\R_+;L^2(\Omega)\big)$.
\end{thm}

\begin{proof}
First of all, we deal with the proof of existence.
We focus here in deriving global in time \tsl{a priori} estimates in $H^1$ for smooth solutions of equation \eqref{eq:PME}. From those estimates,
one can prove the existence of a solution at the claimed level of regularity by, for instance, applying a similar approximation
procedure as the one described in Subsection \ref{ss:local}.

Thus, assume to have a smooth solution $k$ of equation \eqref{eq:PME} on $\R_+\times\Omega$. First of all, by noticing that $k$ in particular
solves
\begin{equation} \label{eq:pme-2}
 \d_tk\,-\,2\,\theta(t)\,\d_x\big(k\,\d_xk\big)\,=\,-\,\rho(t)\,k\,,
\end{equation}
and performing an energy estimate for this equation, we get
\[
\frac{1}{2}\,\frac{d}{dt}\left\|k(t)\right\|_{L^2}^2\,+\,\rho(t)\,\left\|k(t)\right\|_{L^2}^2\,+\,
2\,\theta(t)\,\int_\Omega k\,\left|\d_xk\right|^2\,dx\,=\,0\,. 
\]
This relation immediately implies the following $L^2$ bound:
\begin{equation} \label{est:pme-L^2}
\forall\,t>0\,,\quad \left\|k(t)\right\|_{L^2}^2\,+\,\int^t_0\rho(\t)\,\left\|k(\t)\right\|_{L^2}^2\,d\t\,+\,
2\,\int^t_0\theta(\t)\,\left\|\sqrt{k(\t)}\,\d_xk(\t)\right\|_{L^2}^2\,d\t\,\leq\,\left\|k_0\right\|_{L^2}^2\,,
\end{equation}
where we have replaced the equality by the inequality to be consistent with what one usually gets for weak solutions.

Next, differentiating equation \eqref{eq:pme-2} with respect to space, we find an equation for $\d_xk$:
\[
 \d_t\d_xk\,-\,2\,\theta(t)\,\d_x\big(k\,\d_x^2k\big)\,=\,-\,\rho(t)\,\d_xk\,+\,2\,\theta(t)\,\d_x\left(\big(\d_xk\big)^2\right)\,.
\]
By using the equalities
\begin{align*}
& \int_\Omega\d_x\left(\big(\d_xk\big)^2\right)\,\d_xk\,dx\,=\,2\,\int_\Omega\d_xk\,\d_x^2k\,\d_xk\,dx  \\
&\qquad\qquad\qquad \mbox{ and }\qquad
 \int_\Omega\d_x\left(\big(\d_xk\big)^2\right)\,\d_xk\,dx\,=\,-\,\int_\Omega\big(\d_xk\big)^2\,\d_x^2k\,dx\,,
\end{align*}
which follow respectively by developing the derivative and by integration by parts, we see that the value of the integral on the left is actually zero.
Then, an energy estimate for $\d_xk$ yields
\[
\frac{1}{2}\,\frac{d}{dt}\left\|\d_xk(t)\right\|_{L^2}^2\,+\,\rho(t)\,\left\|\d_xk(t)\right\|_{L^2}^2\,+\,
2\,\theta(t)\,\int_\Omega k\,\left|\dd k\right|^2\,dx\,=\,0\,. 
\]
This relation easily implies a $L^2$ bound for $\d_xk$: we have
\begin{align} \label{est:pme-d_x-L^2}
&\forall\,t>0\,, \\
\nonumber
&\left\|\d_xk(t)\right\|_{L^2}^2\,+\,\int^t_0\rho(\t)\,\left\|\d_xk(\t)\right\|_{L^2}^2\,d\t\,+\,
2\,\int^t_0\theta(\t)\,\left\|\sqrt{k(\t)}\,\d_x^2k(\t)\right\|_{L^2}^2\,d\t\,\leq\,\left\|\d_xk_0\right\|_{L^2}^2\,.
\end{align}
Summing up \eqref{est:pme-L^2} and \eqref{est:pme-d_x-L^2}, we immediately infer the sought global in time $H^1$ estimate for $k$. This completes the proof
of the \tsl{a priori} estimates.

\medbreak
Let us now focus on the uniqueness of solutions at the level of regularity considered in the statement.
Uniqueness will follow from a stability estimate in $L^1$ (in particular, one could formulate a uniqueness statement in a larger functional setting),
which we are going to prove by employing similar techniques as the ones
used for the classical porus medium equation (see \tsl{e.g.} Chapters 4 and 9 of \cite{Va}).
The justification that we can indeed perform those computations at our level of regularity is the main new part of our argument.

Thus, suppose to have two solutions $k_{1,2}$ of equation \eqref{eq:PME}, such that, for $j\in\{1,2\}$, one has
\[
k_j\,\in\,L^\infty\big(\R_+;H^1(\Omega)\big)\qquad \mbox{ and }\qquad
\sqrt{k_j}\,\d_x^2k_j\,\in\,L^2_{\rm loc}\big(\R_+;L^2(\Omega)\big)\,.
\]
Let us set $\de k\,:=\,k_1-k_2$ and, for later use, $\Delta k\,:=\,k_1^2-k_2^2$. We notice that $\de k$ satisfies the equation
\begin{equation} \label{eq:de-k}
\d_t\de k\,-\,\theta(t)\,\dd\left(\Delta k\right)\,=\,-\,\rho(t)\,\de k\,.
\end{equation}
Observe that
\[
 \Delta k\,=\,k_1^2\,-\,k_2^2\,=\,\big(k_1\,+\,k_2\big)\,\de k\qquad \mbox{ belongs to }\quad L^\infty\big(\R_+;H^1(\Omega)\big)\,.
\]
From this property and equation \eqref{eq:de-k}, we deduce that $\d_t\de k\,\in\,L^\infty_{\rm loc}\big(\R_+;H^{-1}(\Omega)\big)$. The local-in-time condition comes
from the presence of the coefficients $\theta$ and $\rho$, for which we have not formulated any boundedness assumption.
However, we need to improve the previous regularity property for the time derivative. For this, we recall that each $k_j$ solves equation
\eqref{eq:pme-2}: by writing
\[
 \d_x\big(k_j\,\d_xk_j\big)\,=\,\sqrt{k_j}\,\sqrt{k_j}\,\dd k_j\,+\,\big(\d_xk_j\big)^2\,,
\]
we see that this quantity belongs to $L^\infty_T(L^1)$ for all $T>0$ fixed. Thus, we get $\d_tk_j\,\in\,L^\infty_T(L^1)$ for all $T>0$, which implies that also
$\d_t\de k$ belongs to this space.

At this point, we take a scalar function $p\,\in\,C^1(\R)$ such that $0\leq p\leq 1$, together with $p(\s)=0$ for $\s\leq0$ and $p'(\s)>0$ for $\s>0$.
Thanks to the previous analysis, we can multiply both sides of equation \eqref{eq:de-k} by $p(\Delta k)$ and integrate over the space domain: we find
\[
\int_\Omega \d_t\de k\;p(\Delta k)\,dx\,=\,-\,\theta(t)\int_\Omega\left|\d_x\Delta k\right|^2\,p'(\Delta k)\,dx\,-\,
\rho(t)\int_\Omega\de k\;p(\Delta k)\,dx\,.
\]
Remark that the sign of $\Delta k$ is the same of $\de k$, owing to the positivity of each $k_j$. Then, the last term on the right-hand side has negative sign.
The same can be said about the first term on the right. Thus, we infer that
\[
 \int_\Omega \d_t\de k\;p(\Delta k)\,dx\,\leq\,0\,.
\]
We now consider a sequence of functions $\big(p_n\big)_n$ verifyng the properties above and which converges to the function $\sign^+$. We recall that
$\sign^+$ is defined as
\[
 \sign^+(\s)=1\quad \mbox{ if }\quad \s>0,\qquad\qquad \sign^+(\s)=0\quad \mbox{ if }\quad \s\leq0\,.
\]
Noticing that $\sign^+(\de k)\,=\,\sign^+(\Delta k)$ and that $\d_t\de k\; \sign^+(\de k)\,=\,\d_t(\de k)_+$ in our functional framework,
from the previous computations we deduce that
\[
\frac{d}{dt}\int_\Omega (\de k)_+\,dx\,\leq\,0\,.
\]
This in particular implies that
\begin{equation} \label{est:pos-part}
\int_\Omega\left(k_1\,-\,k_2\right)_+\,dx\,\leq\,\int_\Omega\left(k_{0,1}\,-\,k_{0,2}\right)_+\,dx\,,
\end{equation}
where, for $j=1,2$, we have denoted by $k_{0,j}$ the initial datum related to the solution $k_j$.

Similar computations for $\wtilde{\de k}\,:=\,k_2-k_1\,=\,-\de k$ yield an estimate analogous to \eqref{est:pos-part} for the negative part $\big(k_1-k_2\big)_-$.
Then, we finally deduce the following $L^1$ stability estimate:
\[
\forall\,t\geq0\,,\qquad\qquad
\left\|k_1(t)\,-\,k_2(t)\right\|_{L^1}\,\leq\,\left\|k_{0,1}\,-\,k_{0,2}\right\|_{L^1}\,.
\]
This latter estimate in particular implies the sought uniqueness result whenever $k_{0,1}\equiv k_{0,2}$, thus completing the proof of the theorem.
\end{proof}


\section*{Declarations}

\paragraph*{Acknowledgements.} The authors are indebted to the anonymous referees, who contributed, through their comments and suggestions,
to clarify the mathematical results and, at the same time, to improve the presentation of the material.

\paragraph*{Ethical approval.} Not applicable.

\paragraph*{Competing interests.} The authors declare to have no competing interests.

\paragraph*{Fundings.}

{
The work of both authors has been partially supported by the project ``TURB1D -- Reduced models of turbulence'', operated by the French CNRS through the program ``International 
Emerging Actions 2019''.

The work of the first author has been partially supported by the LABEX MILYON (ANR-10-LABX-0070) of Universit\'e de Lyon, within the program ``Investissement d'Avenir''
(ANR-11-IDEX-0007), and by the projects BORDS (ANR-16-CE40-0027-01), SingFlows (ANR-18-CE40-0027) and CRISIS (ANR-20-CE40-0020-01), all operated by the French National Research Agency (ANR).

The work of the second author was supported by the project ``Mathematical Analysis of Fluids and Applications'' Grant PID2019-109348GA-I00 funded by MCIN/AEI/ 10.13039/501100011033 and acronym ``MAFyA''. This publication is part of the project PID2019-109348GA-I00 funded by MCIN/ AEI /10.13039/501100011033. R. G-B is also supported by the project "An\'alisis Matem\'atico Aplicado y Ecuaciones Diferenciales" Grant PID2022-141187NB-I00 funded by MCIN/ AEI and acronym "AMAED". This publication is also supported by a 2021 Leonardo Grant for Researchers and Cultural Creators, BBVA Foundation. The BBVA Foundation accepts no responsibility for the opinions, statements, and contents included in the project and/or the results thereof, which are entirely the responsibility of the authors.

}

\paragraph*{Availability of data and materials.} Not applicable.


\end{document}